\documentclass[reqno]{amsart}
\usepackage{amsfonts}

\newtheorem{theorem}{Theorem}
\theoremstyle{plain}

\newtheorem{corollary}{Corollary}

\newtheorem{lemma}{Lemma}

\newtheorem{proposition}{Proposition}

\numberwithin{equation}{section}
\numberwithin{theorem}{section}
\numberwithin{algorithm}{section}
\numberwithin{axiom}{section}
\numberwithin{case}{section}
\numberwithin{claim}{section}
\numberwithin{conclusion}{section}
\numberwithin{condition}{section}
\numberwithin{conjecture}{section}
\numberwithin{corollary}{section}
\numberwithin{criterion}{section}
\numberwithin{definition}{section}
\numberwithin{example}{section}
\numberwithin{exercise}{section}
\numberwithin{lemma}{section}
\numberwithin{notation}{section}
\numberwithin{problem}{section}
\numberwithin{proposition}{section}
\numberwithin{remark}{section}
\numberwithin{solution}{section}

\input{tcilatex}

\begin{document}
\title{Improved Sobolev inequality under constraints}
\author{Fengbo Hang}
\address{Courant Institute, 251 Mercer Street, New York, NY 10012}
\email{fengbo@cims.nyu.edu}
\author{Xiaodong Wang}
\address{Department of Mathematics, Michigan State University, East Lansing,
MI 48824}
\email{xwang@math.msu.edu}

\begin{abstract}
We give a new proof of Aubin's improvement of the Sobolev inequality on $%
\mathbb{S}^{n}$ under the vanishing of first order moments of the area
element and generalize it to higher order moments case. By careful study of
an extremal problem on $\mathbb{S}^{n}$, we determine the constant
explicitly in the second order moments case.
\end{abstract}

\maketitle

\section{Introduction\label{sec1}}

Let $n\in \mathbb{N}$. For $1<p<n$ we denote $p^{\ast }=\frac{pn}{n-p}$. The
classical Sobolev inequality states that there exists a positive constant $%
c\left( n,p\right) $ with%
\begin{equation}
\left\Vert \varphi \right\Vert _{L^{p^{\ast }}}\leq c\left( n,p\right)
\left\Vert \nabla \varphi \right\Vert _{L^{p}}  \label{eq1.1}
\end{equation}%
for any $\varphi \in C_{c}^{\infty }\left( \mathbb{R}^{n}\right) $. Let $%
S_{n,p}$ be the smallest possible choice for the constant $c\left(
n,p\right) $ i.e.%
\begin{eqnarray}
S_{n,p} &=&\sup_{\varphi \in C_{c}^{\infty }\left( \mathbb{R}^{n}\right)
\backslash \left\{ 0\right\} }\frac{\left\Vert \varphi \right\Vert
_{L^{p^{\ast }}}}{\left\Vert \nabla \varphi \right\Vert _{L^{p}}}
\label{eq1.2} \\
&=&\sup \left\{ \frac{\left\Vert u\right\Vert _{L^{p^{\ast }}}}{\left\Vert
\nabla u\right\Vert _{L^{p}}}:u\in L^{p^{\ast }}\left( \mathbb{R}^{n}\right)
\backslash \left\{ 0\right\} \text{ s.t. }\nabla u\in L^{p}\left( \mathbb{R}%
^{n}\right) \right\} .  \notag
\end{eqnarray}%
When no confusion can happen, we write%
\begin{equation}
S_{p}=S_{n,p}.  \label{eq1.3}
\end{equation}%
It follows from \cite{A1, T} that%
\begin{equation}
S_{n,p}=\frac{1}{n}\left( \frac{n\left( p-1\right) }{n-p}\right) ^{1-\frac{1%
}{p}}\left( \frac{n!}{\Gamma \left( \frac{n}{p}\right) \Gamma \left( n+1-%
\frac{n}{p}\right) \left\vert \mathbb{S}^{n-1}\right\vert }\right) ^{\frac{1%
}{n}}  \label{eq1.4}
\end{equation}%
and it is achieved if and only if%
\begin{equation}
u=\pm \left( a+b\left\vert x-x_{0}\right\vert ^{\frac{p}{p-1}}\right) ^{1-%
\frac{n}{p}}  \label{eq1.5}
\end{equation}%
for some $a>0$, $b>0$ and $x_{0}\in \mathbb{R}^{n}$. Here%
\begin{equation}
\Gamma \left( \alpha \right) =\int_{0}^{\infty }t^{\alpha -1}e^{-t}dt
\label{eq1.6}
\end{equation}%
for $\alpha >0$. When $p=2$, the sharp constant $S_{2}$ is closely related
to the classical Yamabe problem (\cite{He2, SY}).

Let $\left( M^{n},g\right) $ be a smooth compact Riemannian manifold of
dimension $n$ and $1<p<n$. Aubin \cite{A1} showed that for any $\varepsilon
>0$, we have%
\begin{equation}
\left\Vert u\right\Vert _{L^{p^{\ast }}\left( M\right) }^{p}\leq \left(
S_{p}^{p}+\varepsilon \right) \left\Vert \nabla u\right\Vert _{L^{p}\left(
M\right) }^{p}+c\left( \varepsilon \right) \left\Vert u\right\Vert
_{L^{p}\left( M\right) }^{p}  \label{eq1.7}
\end{equation}%
for $u\in W^{1,p}\left( M\right) $.

On the other hand, in \cite{A2} he further proved that for any $\varepsilon
>0$ and $u\in W^{1,p}\left( \mathbb{S}^{n}\right) $ with%
\begin{equation}
\int_{\mathbb{S}^{n}}x_{i}\left\vert u\right\vert ^{p^{\ast }}d\mu \left(
x\right) =0\text{ for }i=1,2,\cdots ,n+1,  \label{eq1.8}
\end{equation}%
here $\mu $ is the measure associated with the standard metric on $\mathbb{S}%
^{n}$, we have%
\begin{equation}
\left\Vert u\right\Vert _{L^{p^{\ast }}\left( \mathbb{S}^{n}\right)
}^{p}\leq \left( 2^{-\frac{p}{n}}S_{p}^{p}+\varepsilon \right) \left\Vert
\nabla u\right\Vert _{L^{p}\left( \mathbb{S}^{n}\right) }^{p}+c\left(
\varepsilon \right) \left\Vert u\right\Vert _{L^{p}\left( \mathbb{S}%
^{n}\right) }^{p}.  \label{eq1.9}
\end{equation}%
In other words, the constant $S_{p}^{p}+\varepsilon $ in (\ref{eq1.7}) can
be improved to $2^{-\frac{p}{n}}S_{p}^{p}+\varepsilon $ when the first order
moments of $\left\vert u\right\vert ^{p^{\ast }}$ are all zeros. (\ref{eq1.9}%
) is closely related to the problem of prescribing scalar curvature on $%
\mathbb{S}^{n}$ (see for example \cite{A2, ChY, He1}). The main aim of this
note is to generalize (\ref{eq1.9}) to higher order moments case. To state
our results we introduce some notations. For a nonnegative integer $k$, we
denote%
\begin{eqnarray}
\mathcal{P}_{k} &=&\left\{ \text{all polynomials on }\mathbb{R}^{n+1}\text{
with degree at most }k\right\} ;  \label{eq1.10} \\
\overset{\circ }{\mathcal{P}}_{k} &=&\left\{ f\in \mathcal{P}_{k}:\int_{%
\mathbb{S}^{n}}fd\mu =0\right\} .  \label{eq1.11}
\end{eqnarray}%
Here $\mu $ is the standard measure on $\mathbb{S}^{n}$.

For $m\in \mathbb{N}$, we denote%
\begin{eqnarray}
&&\mathcal{M}_{m}^{c}\left( \mathbb{S}^{n}\right)  \label{eq1.12} \\
&=&\left\{ \nu :\nu \text{ is a probability measure on }\mathbb{S}^{n}\text{
supported on countably}\right.  \notag \\
&&\left. \text{many points s.t. }\int_{\mathbb{S}^{n}}fd\nu =0\text{ for all 
}f\in \overset{\circ }{\mathcal{P}}_{m}\right\}  \notag \\
&=&\left\{ \nu :\nu \text{ is a probability measure on }\mathbb{S}^{n}\text{
supported on countably}\right.  \notag \\
&&\left. \text{many points s.t. }\frac{1}{\left\vert \mathbb{S}%
^{n}\right\vert }\int_{\mathbb{S}^{n}}fd\mu =\int_{\mathbb{S}^{n}}fd\nu 
\text{ for all }f\in \mathcal{P}_{m}\right\} .  \notag
\end{eqnarray}%
For $0<\theta <1$, we define%
\begin{eqnarray}
&&\Theta \left( m,\theta ,n\right)  \label{eq1.13} \\
&=&\inf \left\{ \sum_{i}\nu _{i}^{\theta }:\nu \in \mathcal{M}_{m}^{c}\left( 
\mathbb{S}^{n}\right) \text{ is supported on }\left\{ \xi _{i}\right\}
\subset \mathbb{S}^{n}\text{, }\nu _{i}=\nu \left( \left\{ \xi _{i}\right\}
\right) \right\} .  \notag
\end{eqnarray}

\begin{theorem}
\label{thm1.1}Assume $n\in \mathbb{N}$, $1<p<n$ and $m\in \mathbb{N}$.
Denote $p^{\ast }=\frac{pn}{n-p}$. Then for any $\varepsilon >0$, and $u\in
W^{1,p}\left( \mathbb{S}^{n}\right) $ with%
\begin{equation}
\int_{\mathbb{S}^{n}}f\left\vert u\right\vert ^{p^{\ast }}d\mu =0
\label{eq1.14}
\end{equation}%
for all $f\in \overset{\circ }{\mathcal{P}}_{m}$ (defined in (\ref{eq1.11}%
)), we have%
\begin{equation}
\left\Vert u\right\Vert _{L^{p^{\ast }}\left( \mathbb{S}^{n}\right)
}^{p}\leq \left( \frac{S_{n,p}^{p}}{\Theta \left( m,\frac{n-p}{n},n\right) }%
+\varepsilon \right) \left\Vert \nabla u\right\Vert _{L^{p}\left( \mathbb{S}%
^{n}\right) }^{p}+c_{\varepsilon }\left\Vert u\right\Vert _{L^{p}\left( 
\mathbb{S}^{n}\right) }^{p}.  \label{eq1.15}
\end{equation}%
Here $S_{n,p}$ is given by (\ref{eq1.4}), $\Theta \left( m,\frac{n-p}{n}%
,n\right) $ is given by (\ref{eq1.13}).
\end{theorem}

We will prove in Proposition \ref{prop3.1} that $\Theta \left( 1,\theta
,n\right) =2^{1-\theta }$. In particular, (\ref{eq1.9}) follows from Theorem %
\ref{thm1.1}. In Proposition \ref{prop3.2} we will show $\Theta \left(
2,\theta ,n\right) =\left( n+2\right) ^{1-\theta }$. Hence we have

\begin{corollary}
\label{cor1.1}Assume $n\in \mathbb{N}$ and $1<p<n$. Denote $p^{\ast }=\frac{%
pn}{n-p}$. Then for any $\varepsilon >0$, and $u\in W^{1,p}\left( \mathbb{S}%
^{n}\right) $ with%
\begin{equation}
\int_{\mathbb{S}^{n}}f\left\vert u\right\vert ^{p^{\ast }}d\mu =0
\label{eq1.16}
\end{equation}%
for all $f\in \overset{\circ }{\mathcal{P}}_{2}$ (defined in (\ref{eq1.11}%
)), we have%
\begin{equation}
\left\Vert u\right\Vert _{L^{p^{\ast }}\left( \mathbb{S}^{n}\right)
}^{p}\leq \left( \frac{S_{n,p}^{p}}{\left( n+2\right) ^{\frac{p}{n}}}%
+\varepsilon \right) \left\Vert \nabla u\right\Vert _{L^{p}\left( \mathbb{S}%
^{n}\right) }^{p}+c_{\varepsilon }\left\Vert u\right\Vert _{L^{p}\left( 
\mathbb{S}^{n}\right) }^{p}.  \label{eq1.17}
\end{equation}%
Here $S_{n,p}$ is given by (\ref{eq1.4}).
\end{corollary}

In Lemma \ref{lem4.1} we will show the constant $\frac{S_{n,p}^{p}}{\left(
n+2\right) ^{\frac{p}{n}}}+\varepsilon $ is almost optimal. It remains an
interesting open question to find the exact value of $\Theta \left( m,\theta
,n\right) $ for $m\geq 3$.

Our derivation of Theorem \ref{thm1.1} is motivated by recent work in \cite%
{ChH, Ha}, where Aubin's Moser-Trudinger-Onofri inequality on $\mathbb{S}%
^{n} $ in \cite{A2} is generalized to higher order moments cases. The
sequence of inequalities in \cite{ChH, Ha} are motivated by similar
inequalities on $\mathbb{S}^{1}$ (see \cite{GrS, OPS, W}). The crucial point
of \cite{ChH, Ha} is some new developments of concentration compactness
principle for Sobolev spaces in critical dimensions, i.e. for Sobolev
functions in $W^{1,n}\left( M^{n},g\right) $ (see \cite{ChH, Ha, Ln1, Ln2}).
These developments provide a new proof of Aubin's first order moment
inequality and make the derivation of higher order moments case possible. On
the other hand, the concentration compactness principle in subcritical case
(i.e. for Sobolev functions in $W^{1,p}\left( M^{n},g\right) $, $1<p<n$) had
been well understood in \cite{He2, Ln1, Ln2}. We will apply it to prove
Theorem \ref{thm1.1}.

In Section \ref{sec2}, we recall the concentration compactness principle and
apply it to derive Theorem \ref{thm1.1}. In Section \ref{sec3}, we study the
extremal problem on $\mathbb{S}^{n}$ associated with Theorem \ref{thm1.1}
and compute $\Theta \left( m,\theta ,n\right) $ in several cases. In Section %
\ref{sec4}, we first show the coefficients of $\left\Vert \nabla
u\right\Vert _{L^{p}}^{p}$ in Corollary \ref{cor1.1} are almost optimal,
then we discuss generalizations of Theorem \ref{thm1.1} to functions in
higher order Sobolev spaces.

We would like to thank the referee for careful reading of the manuscript and
many helpful suggestions.

\section{Concentration compactness principle\label{sec2}}

We first recall the concentration compactness principle (\cite{He2, Ln1, Ln2}%
). It is closely related to Aubin's almost sharp inequality (\ref{eq1.7}).

\begin{theorem}
\label{thm2.1}Let $\left( M^{n},g\right) $ be a smooth compact Riemannian
manifold. Assume $1<p<n$, $u_{i}\in W^{1,p}\left( M\right) $ such that $%
u_{i}\rightharpoonup u$ weakly in $W^{1,p}\left( M\right) $,%
\begin{eqnarray}
\left\vert \nabla u_{i}\right\vert ^{p}d\mu &\rightarrow &\left\vert \nabla
u\right\vert ^{p}d\mu +\sigma \text{ as measure,}  \label{eq2.1} \\
\left\vert u_{i}\right\vert ^{p^{\ast }}d\mu &\rightarrow &\left\vert
u\right\vert ^{p^{\ast }}d\mu +\nu \text{ as measure,}  \label{eq2.2}
\end{eqnarray}%
here $\mu $ is the measure associated with metric $g$, then we can find
countably many points $x_{i}\in M$ such that%
\begin{eqnarray}
\nu &=&\sum_{i}\nu _{i}\delta _{x_{i}},  \label{eq2.3} \\
\nu _{i}^{\frac{1}{p^{\ast }}} &\leq &S_{p}\sigma _{i}^{\frac{1}{p}},
\label{eq2.4}
\end{eqnarray}%
here $\nu _{i}=\nu \left( \left\{ x_{i}\right\} \right) $, $\sigma
_{i}=\sigma \left( \left\{ x_{i}\right\} \right) $.
\end{theorem}

We will prove Theorem \ref{thm1.1} using the concentration compactness
principle in the same spirit as \cite{ChH, Ha}. As a warm up, we derive a
small variation of \cite[Theorem 1]{A2} by this approach.

\begin{proposition}
\label{prop2.1}Let $\left( M^{n},g\right) $ be a smooth compact Riemannian
manifold with dimension $n$ and $1<p<n$, $p^{\ast }=\frac{pn}{n-p}$. For $%
1\leq i<l$ ($l$ can be $\infty $), we have $f_{i}\in C\left( M,\mathbb{R}%
\right) $ such that for some positive constant $\lambda $ and $\Lambda $,%
\begin{equation}
\lambda \leq \sum_{1\leq i<l}\left\vert f_{i}\right\vert ^{\frac{n-p}{p}%
}\leq \Lambda .  \label{eq2.5}
\end{equation}%
Then for any $\varepsilon >0$ and $u\in W^{1,p}\left( M\right) $ with%
\begin{equation}
\left\vert \int_{M}f_{i}\left\vert u\right\vert ^{p^{\ast }}d\mu \right\vert
\leq b_{i}\left\Vert u\right\Vert _{L^{p}}^{p^{\ast }}\text{ for }1\leq i<l,
\label{eq2.6}
\end{equation}%
here $b_{i}$ is a nonnegative number depending only on $i$, we have%
\begin{equation}
\left\Vert u\right\Vert _{L^{p^{\ast }}}^{p}\leq \left( 2^{-\frac{p}{n}}%
\frac{\Lambda }{\lambda }S_{n,p}^{p}+\varepsilon \right) \left\Vert \nabla
u\right\Vert _{L^{p}\left( M\right) }^{p}+c\left( \varepsilon
,b_{i},f_{i}\right) \left\Vert u\right\Vert _{L^{p}\left( M\right) }^{p}.
\label{eq2.7}
\end{equation}%
Here $S_{n,p}$ is given by (\ref{eq1.4}).
\end{proposition}

\begin{proof}
Let $\alpha =2^{-\frac{p}{n}}\frac{\Lambda }{\lambda }S_{p}^{p}+\varepsilon $%
, here $S_{p}=S_{n,p}$. Assume (\ref{eq2.7}) is not true, then for any $k\in 
\mathbb{N}$, we can find a $u_{k}\in W^{1,p}\left( M\right) $ with%
\begin{equation}
\left\vert \int_{M}f_{i}\left\vert u_{k}\right\vert ^{p^{\ast }}d\mu
\right\vert \leq b_{i}\left\Vert u_{k}\right\Vert _{L^{p}}^{p^{\ast }}\text{
for }1\leq i<l  \label{eq2.8}
\end{equation}%
and%
\begin{equation}
\left\Vert u_{k}\right\Vert _{L^{p^{\ast }}}^{p}>\alpha \left\Vert \nabla
u_{k}\right\Vert _{L^{p}\left( M\right) }^{p}+k\left\Vert u_{k}\right\Vert
_{L^{p}\left( M\right) }^{p}.  \label{eq2.9}
\end{equation}%
By scaling we can assume%
\begin{equation}
\left\Vert u_{k}\right\Vert _{L^{p^{\ast }}}=1.  \label{eq2.10}
\end{equation}%
We have%
\begin{equation}
\left\Vert \nabla u_{k}\right\Vert _{L^{p}\left( M\right) }^{p}\leq \frac{1}{%
\alpha }\text{ and }\left\Vert u_{k}\right\Vert _{L^{p}\left( M\right)
}^{p}\leq \frac{1}{k}.  \label{eq2.11}
\end{equation}%
It follows that $u_{k}\rightharpoonup 0$ weakly in $W^{1,p}\left( M\right) $%
. After passing to a subsequence (still denoted as $u_{k}$) we have%
\begin{equation}
\left\vert \nabla u_{k}\right\vert ^{p}d\mu \rightarrow \sigma \text{ and }%
\left\vert u_{k}\right\vert ^{p^{\ast }}d\mu \rightarrow \nu  \label{eq2.12}
\end{equation}%
as measure. It follows from Theorem \ref{thm2.1} that we can find countably
many points $x_{j}\in M$ such that%
\begin{equation}
\nu =\sum_{j}\nu _{j}\delta _{x_{j}}\text{ and }\nu _{j}^{\frac{1}{p^{\ast }}%
}\leq S_{p}\sigma _{j}^{\frac{1}{p}}.  \label{eq2.13}
\end{equation}%
Here $\nu _{j}=\nu \left( \left\{ x_{j}\right\} \right) $ and $\sigma
_{j}=\sigma \left( \left\{ x_{j}\right\} \right) $. For convenience we denote%
\begin{equation}
\theta =\frac{p}{p^{\ast }}=\frac{n-p}{n}.  \label{eq2.14}
\end{equation}%
We have%
\begin{equation}
\sum_{j}\nu _{j}=\nu \left( M\right) =1  \label{eq2.15}
\end{equation}%
and%
\begin{equation}
\sum_{j}\nu _{j}^{\theta }\leq S_{p}^{p}\sum_{j}\sigma _{j}\leq
S_{p}^{p}\sigma \left( M\right) \leq \frac{S_{p}^{p}}{\alpha }.
\label{eq2.16}
\end{equation}%
It follows from (\ref{eq2.8}) and (\ref{eq2.11}) that for $1\leq i<l$,%
\begin{equation*}
\int_{M}f_{i}d\nu =0.
\end{equation*}%
In other words,%
\begin{equation}
\sum_{j}f_{i}\left( x_{j}\right) \nu _{j}=0.  \label{eq2.17}
\end{equation}%
Hence%
\begin{equation*}
\sum_{j,f_{i}\left( x_{j}\right) \geq 0}\left\vert f_{i}\left( x_{j}\right)
\right\vert \nu _{j}=\sum_{j,f_{i}\left( x_{j}\right) <0}\left\vert
f_{i}\left( x_{j}\right) \right\vert \nu _{j}.
\end{equation*}%
We have%
\begin{equation*}
\left( \sum_{j}\left\vert f_{i}\left( x_{j}\right) \right\vert \nu
_{j}\right) ^{\theta }=2^{\theta }\left( \sum_{j,f_{i}\left( x_{j}\right)
\geq 0}\left\vert f_{i}\left( x_{j}\right) \right\vert \nu _{j}\right)
^{\theta }\leq 2^{\theta }\sum_{j,f_{i}\left( x_{j}\right) \geq 0}\left\vert
f_{i}\left( x_{j}\right) \right\vert ^{\theta }\nu _{j}^{\theta }
\end{equation*}%
and%
\begin{equation*}
\left( \sum_{j}\left\vert f_{i}\left( x_{j}\right) \right\vert \nu
_{j}\right) ^{\theta }=2^{\theta }\left( \sum_{j,f_{i}\left( x_{j}\right)
<0}\left\vert f_{i}\left( x_{j}\right) \right\vert \nu _{j}\right) ^{\theta
}\leq 2^{\theta }\sum_{j,f_{i}\left( x_{j}\right) <0}\left\vert f_{i}\left(
x_{j}\right) \right\vert ^{\theta }\nu _{j}^{\theta }.
\end{equation*}%
Summing up these two inequalities we get%
\begin{equation}
\left( \sum_{j}\left\vert f_{i}\left( x_{j}\right) \right\vert \nu
_{j}\right) ^{\theta }\leq 2^{\theta -1}\sum_{j}\left\vert f_{i}\left(
x_{j}\right) \right\vert ^{\theta }\nu _{j}^{\theta }.  \label{eq2.18}
\end{equation}%
To continue we observe that by Holder inequality,%
\begin{equation*}
\sum_{j}\left\vert f_{i}\left( x_{j}\right) \right\vert ^{\theta }\nu
_{j}\leq \left( \sum_{j}\left\vert f_{i}\left( x_{j}\right) \right\vert \nu
_{j}\right) ^{\theta }\left( \sum_{j}\nu _{j}\right) ^{1-\theta }=\left(
\sum_{j}\left\vert f_{i}\left( x_{j}\right) \right\vert \nu _{j}\right)
^{\theta },
\end{equation*}%
hence%
\begin{equation*}
\sum_{j}\left\vert f_{i}\left( x_{j}\right) \right\vert ^{\theta }\nu
_{j}\leq 2^{\theta -1}\sum_{j}\left\vert f_{i}\left( x_{j}\right)
\right\vert ^{\theta }\nu _{j}^{\theta }.
\end{equation*}%
Summing up over $i$, we get%
\begin{equation*}
\sum_{1\leq i<l}\sum_{j}\left\vert f_{i}\left( x_{j}\right) \right\vert
^{\theta }\nu _{j}\leq 2^{\theta -1}\sum_{1\leq i<l}\sum_{j}\left\vert
f_{i}\left( x_{j}\right) \right\vert ^{\theta }\nu _{j}^{\theta }.
\end{equation*}%
Using (\ref{eq2.5}), (\ref{eq2.15}) and (\ref{eq2.16}), we get%
\begin{equation*}
\lambda \leq 2^{\theta -1}\Lambda \sum_{j}\nu _{j}^{\theta }\leq 2^{\theta
-1}\Lambda \frac{S_{p}^{p}}{\alpha }.
\end{equation*}%
Hence%
\begin{equation*}
\alpha \leq 2^{-\frac{p}{n}}\frac{\Lambda }{\lambda }S_{p}^{p}.
\end{equation*}%
This contradicts with the definition of $\alpha $.
\end{proof}

Now let us turn to Theorem \ref{thm1.1}. Indeed we will prove the following
slightly more general

\begin{theorem}
\label{thm2.2}Assume $n\in \mathbb{N}$, $1<p<n$ and $m\in \mathbb{N}$.
Denote $p^{\ast }=\frac{pn}{n-p}$. Then for any $\varepsilon >0$, and $u\in
W^{1,p}\left( \mathbb{S}^{n}\right) $ with%
\begin{equation}
\left\vert \int_{\mathbb{S}^{n}}f\left\vert u\right\vert ^{p^{\ast }}d\mu
\right\vert \leq b\left( f\right) \left\Vert u\right\Vert _{L^{p}}^{p^{\ast
}}  \label{eq2.19}
\end{equation}%
for all $f\in \overset{\circ }{\mathcal{P}}_{m}$ (defined in (\ref{eq1.11}%
)), here $b\left( f\right) $ is a positive number depending on $f$, we have%
\begin{equation}
\left\Vert u\right\Vert _{L^{p^{\ast }}\left( \mathbb{S}^{n}\right)
}^{p}\leq \left( \frac{S_{n,p}^{p}}{\Theta \left( m,\frac{n-p}{n},n\right) }%
+\varepsilon \right) \left\Vert \nabla u\right\Vert _{L^{p}\left( \mathbb{S}%
^{n}\right) }^{p}+c\left( \varepsilon ,b,m\right) \left\Vert u\right\Vert
_{L^{p}\left( \mathbb{S}^{n}\right) }^{p}.  \label{eq2.20}
\end{equation}%
Here $S_{n,p}$ is given by (\ref{eq1.4}), $\Theta \left( m,\frac{n-p}{n}%
,n\right) $ is given by (\ref{eq1.13}).
\end{theorem}

\begin{proof}
We will proceed in the same way as in the proof of Proposition \ref{prop2.1}%
. Let $S_{p}=S_{n,p}$ and%
\begin{equation}
\alpha =\frac{S_{p}^{p}}{\Theta \left( m,\frac{n-p}{n},n\right) }%
+\varepsilon .  \label{eq2.21}
\end{equation}%
If (\ref{eq2.20}) is not true, then for any $k\in \mathbb{N}$, we can find a 
$u_{k}\in W^{1,p}\left( \mathbb{S}^{n}\right) $ with%
\begin{equation}
\left\vert \int_{M}f\left\vert u_{k}\right\vert ^{p^{\ast }}d\mu \right\vert
\leq b\left( f\right) \left\Vert u_{k}\right\Vert _{L^{p}}^{p^{\ast }}\text{
for }f\in \overset{\circ }{\mathcal{P}}_{m}  \label{eq2.22}
\end{equation}%
and%
\begin{equation}
\left\Vert u_{k}\right\Vert _{L^{p^{\ast }}}^{p}>\alpha \left\Vert \nabla
u_{k}\right\Vert _{L^{p}\left( M\right) }^{p}+k\left\Vert u_{k}\right\Vert
_{L^{p}\left( M\right) }^{p}.  \label{eq2.23}
\end{equation}%
We can assume%
\begin{equation}
\left\Vert u_{k}\right\Vert _{L^{p^{\ast }}}=1.  \label{eq2.24}
\end{equation}%
Then%
\begin{equation}
\left\Vert \nabla u_{k}\right\Vert _{L^{p}}^{p}\leq \frac{1}{\alpha }\text{
and }\left\Vert u_{k}\right\Vert _{L^{p}}^{p}\leq \frac{1}{k}.
\label{eq2.25}
\end{equation}%
It follows that $u_{k}\rightharpoonup 0$ weakly in $W^{1,p}\left( \mathbb{S}%
^{n}\right) $. After passing to a subsequence (still denoted as $u_{k}$) we
have%
\begin{equation}
\left\vert \nabla u_{k}\right\vert ^{p}d\mu \rightarrow \sigma \text{ and }%
\left\vert u_{k}\right\vert ^{p^{\ast }}d\mu \rightarrow \nu  \label{eq2.26}
\end{equation}%
as measure. By Theorem \ref{thm2.1} we can find countably many points $\xi
_{j}\in \mathbb{S}^{n}$ such that%
\begin{equation}
\nu =\sum_{j}\nu _{j}\delta _{\xi _{j}}\text{ and }\nu _{j}^{\frac{1}{%
p^{\ast }}}\leq S_{p}\sigma _{j}^{\frac{1}{p}}.  \label{eq2.27}
\end{equation}%
Here $\nu _{j}=\nu \left( \left\{ \xi _{j}\right\} \right) $ and $\sigma
_{j}=\sigma \left( \left\{ \xi _{j}\right\} \right) $. Let%
\begin{equation}
\theta =\frac{p}{p^{\ast }}=\frac{n-p}{n}.  \label{eq2.28}
\end{equation}%
Then 
\begin{equation}
\nu \left( \mathbb{S}^{n}\right) =1\text{ and }\sigma \left( \mathbb{S}%
^{n}\right) \leq \frac{1}{\alpha }.  \label{eq2.29}
\end{equation}%
It follows from (\ref{eq2.22}) and (\ref{eq2.25}) that $\int_{\mathbb{S}%
^{n}}fd\nu =0$ for $f\in \overset{\circ }{\mathcal{P}}_{m}$, hence $\nu \in 
\mathcal{M}_{m}^{c}\left( \mathbb{S}^{n}\right) $ (see (\ref{eq1.12})). By
definition of $\Theta \left( m,\theta ,n\right) $ (see (\ref{eq1.13})), (\ref%
{eq2.27}) and (\ref{eq2.29}) we have%
\begin{equation*}
\Theta \left( m,\theta ,n\right) \leq \sum_{j}\nu _{j}^{\theta }\leq
\sum_{j}S_{p}^{p}\sigma _{j}\leq S_{p}^{p}\sigma \left( \mathbb{S}%
^{n}\right) \leq \frac{S_{p}^{p}}{\alpha }.
\end{equation*}%
Hence%
\begin{equation*}
\alpha \leq \frac{S_{p}^{p}}{\Theta \left( m,\theta ,n\right) }.
\end{equation*}%
This inequality contradicts with the choice of $\alpha $ in (\ref{eq2.21}).
\end{proof}

\section{An extremal problem on the sphere\label{sec3}}

The main aim of this section is to compute $\Theta \left( 1,\theta ,n\right) 
$, $\Theta \left( 2,\theta ,n\right) $ and $\Theta \left( m,\theta ,1\right) 
$. It remains an interesting problem to find $\Theta \left( m,\theta
,n\right) $ for all $m,n\in \mathbb{N}$ and $\theta \in \left( 0,1\right) $.

\begin{proposition}
\label{prop3.1}For $\theta \in \left( 0,1\right) $ and $n\in \mathbb{N}$,%
\begin{equation}
\Theta \left( 1,\theta ,n\right) =2^{1-\theta }.  \label{eq3.1}
\end{equation}%
Moreover $\Theta \left( 1,\theta ,n\right) $ is achieved at $\nu \in 
\mathcal{M}_{1}^{c}\left( \mathbb{S}^{n}\right) $ if and only if $\nu =\frac{%
1}{2}\delta _{\xi }+\frac{1}{2}\delta _{-\xi }$ for some $\xi \in \mathbb{S}%
^{n}$.
\end{proposition}

To prove Proposition \ref{prop3.1}, we first treat the case the measure $\nu 
$ is supported on finitely many points. For $N\in \mathbb{N}$, $N\geq 2$, we
define%
\begin{eqnarray}
W_{N} &=&\left\{ \left( \alpha _{1},\cdots ,\alpha _{N}\right) :\alpha
_{i}\geq 0\text{ for }1\leq i\leq N\text{, }\sum_{i=1}^{N}\alpha
_{i}=1,\right.  \label{eq3.2} \\
&&\left. \exists \xi _{1},\cdots ,\xi _{N}\in \mathbb{S}^{n}\text{ s.t. }%
\alpha _{1}\xi _{1}+\cdots +\alpha _{N}\xi _{N}=0\right\} .  \notag
\end{eqnarray}%
It is worth pointing out that $\xi _{1},\cdots ,\xi _{N}$ do not need to be
mutually different.\ It is clear that $W_{N}$ is a compact subset of $%
\mathbb{R}^{N}$.

\begin{lemma}
\label{lem3.1}For $0<\theta <1$ and $\left( \alpha _{1},\cdots ,\alpha
_{N}\right) \in W_{N}$, we have%
\begin{equation}
\alpha _{1}^{\theta }+\cdots +\alpha _{N}^{\theta }\geq 2^{1-\theta }.
\label{eq3.3}
\end{equation}%
Equality holds if and only if $\left( \alpha _{1},\cdots ,\alpha _{N}\right)
=\left( \frac{1}{2},\frac{1}{2},0,\cdots ,0\right) $ after permutation.
\end{lemma}

\begin{proof}
For $w\in W_{N}$, $w=\left( \alpha _{1},\cdots ,\alpha _{N}\right) $, we
denote%
\begin{equation}
f\left( w\right) =\alpha _{1}^{\theta }+\cdots +\alpha _{N}^{\theta }.
\label{eq3.4}
\end{equation}%
We have%
\begin{equation}
f\left( w\right) \geq \left( \alpha _{1}+\cdots +\alpha _{N}\right) ^{\theta
}=1.  \label{eq3.5}
\end{equation}%
Let%
\begin{equation}
\lambda =\min_{W_{N}}f=f\left( w_{0}\right)  \label{eq3.6}
\end{equation}%
for some $w_{0}\in W_{N}$. Since $\left( \frac{1}{2},\frac{1}{2},0,\cdots
,0\right) \in W_{N}$, we see%
\begin{equation}
\lambda \leq f\left( \frac{1}{2},\frac{1}{2},0,\cdots ,0\right) =2^{1-\theta
}.  \label{eq3.7}
\end{equation}%
After permutation, we assume%
\begin{equation}
\alpha _{1}>0,\alpha _{2}>0,\cdots ,\alpha _{k}>0,\alpha _{k+1}=0,\cdots
,\alpha _{N}=0.  \label{eq3.8}
\end{equation}

If $k=2$, then we can find $\xi _{1},\xi _{2}\in \mathbb{S}^{n}$ s.t. $%
\alpha _{1}\xi _{1}+\alpha _{2}\xi _{2}=0$. It follows that $\alpha
_{1}=\alpha _{2}$. Using $\alpha _{1}+\alpha _{2}=1$, we see $\alpha
_{1}=\alpha _{2}=\frac{1}{2}$ and $\xi _{2}=-\xi _{1}$. Hence $\lambda
=2^{1-\theta }$.

Next we want to show $k$ can not be larger than $2$. Indeed if $k\geq 3$, we
can find $\xi _{1},\cdots ,\xi _{k}\in \mathbb{S}^{n}$ s.t.%
\begin{equation}
\alpha _{1}\xi _{1}+\cdots +\alpha _{k}\xi _{k}=0.  \label{eq3.9}
\end{equation}%
We first observe that for $i\neq j$, we must have $\xi _{i}\neq \xi _{j}$.
If this is not the case, say $\xi _{1}=\xi _{2}$, then%
\begin{equation*}
\left( \alpha _{1}+\alpha _{2},\alpha _{3},\cdots ,\alpha _{k},0,\cdots
,0\right) \in W_{N}.
\end{equation*}%
Since $0<\theta <1$,%
\begin{eqnarray*}
&&f\left( \alpha _{1}+\alpha _{2},\alpha _{3},\cdots ,\alpha _{k},0,\cdots
,0\right) \\
&=&\left( \alpha _{1}+\alpha _{2}\right) ^{\theta }+\alpha _{3}^{\theta
}+\cdots +\alpha _{k}^{\theta } \\
&<&\alpha _{1}^{\theta }+\alpha _{2}^{\theta }+\cdots +\alpha _{k}^{\theta }
\\
&=&\lambda ,
\end{eqnarray*}%
this contradicts with the definition of $\lambda $ (see (\ref{eq3.6})).

To continue, we observe that for some $1\leq i\leq k$, we have $\alpha
_{i}^{\theta -1}\neq \lambda $. Indeed if $\alpha _{i}^{\theta -1}=\lambda $
for all $i\leq k$, then $\alpha _{i}=\frac{1}{k}$ for $i\leq k$ and $\lambda
=f\left( w_{0}\right) =k^{1-\theta }$, this contradicts with (\ref{eq3.7}).

By permutation we assume $\alpha _{k}^{\theta -1}\neq \lambda $. We will
show for $i\leq k-1$, $\xi _{i}=\pm \xi _{k}$. Indeed we can assume $i=1$.
Let $\eta \in \mathbb{S}^{n}$ satisfy $\left\langle \xi _{1},\eta
\right\rangle =0$, here $\left\langle \cdot ,\cdot \right\rangle $ denotes
the standard inner product in $\mathbb{R}^{n+1}$. We consider%
\begin{eqnarray}
x &=&x\left( t\right)  \label{eq3.10} \\
&=&-\alpha _{1}\left( \cos t\cdot \xi _{1}+\sin t\cdot \eta \right) -\alpha
_{2}\xi _{2}-\cdots -\alpha _{k-1}\xi _{k-1},  \notag
\end{eqnarray}%
then%
\begin{equation}
x\left( 0\right) =-\alpha _{1}\xi _{2}-\cdots -\alpha _{k-1}\xi
_{k-1}=\alpha _{k}\xi _{k}\neq 0.  \label{eq3.11}
\end{equation}%
Hence $x\neq 0$ for $t$ near $0$. We have%
\begin{equation}
\alpha _{1}\left( \cos t\cdot \xi _{1}+\sin t\cdot \eta \right) +\alpha
_{2}\xi _{2}+\cdots +\alpha _{k-1}\xi _{k-1}+\left\vert x\right\vert \frac{x%
}{\left\vert x\right\vert }=0.  \label{eq3.12}
\end{equation}%
Note that%
\begin{equation*}
\alpha _{1}+\cdots +\alpha _{k-1}+\left\vert x\right\vert =1-\alpha
_{k}+\left\vert x\right\vert ,
\end{equation*}%
hence%
\begin{equation}
\left( \frac{\alpha _{1}}{1-\alpha _{k}+\left\vert x\right\vert },\cdots ,%
\frac{\alpha _{k-1}}{1-\alpha _{k}+\left\vert x\right\vert },\frac{%
\left\vert x\right\vert }{1-\alpha _{k}+\left\vert x\right\vert },0,\cdots
,0\right) \in W_{N}.  \label{eq3.13}
\end{equation}%
It follows that%
\begin{equation*}
\lambda \leq \frac{\alpha _{1}^{\theta }+\cdots +\alpha _{k-1}^{\theta
}+\left\vert x\right\vert ^{\theta }}{\left( 1-\alpha _{k}+\left\vert
x\right\vert \right) ^{\theta }}=\frac{\lambda -\alpha _{k}^{\theta
}+\left\vert x\right\vert ^{\theta }}{\left( 1-\alpha _{k}+\left\vert
x\right\vert \right) ^{\theta }}.
\end{equation*}%
In particular,%
\begin{equation*}
0=\left. \frac{d}{dt}\right\vert _{t=0}\frac{\lambda -\alpha _{k}^{\theta
}+\left\vert x\right\vert ^{\theta }}{\left( 1-\alpha _{k}+\left\vert
x\right\vert \right) ^{\theta }}=\theta \alpha _{1}\left( \lambda -\alpha
_{k}^{\theta -1}\right) \left\langle \xi _{k},\eta \right\rangle .
\end{equation*}%
We deduce $\left\langle \xi _{k},\eta \right\rangle =0$, using $\alpha
_{1}>0 $ and $\alpha _{k}^{\theta -1}\neq \lambda $. Hence $\xi _{1}=\pm \xi
_{k}$.

Since $k\geq 3$, we know for some $1\leq i<j\leq k$, $\xi _{i}=\xi _{j}$.
This gives us a contradiction with our previous observation.
\end{proof}

Let%
\begin{eqnarray}
W_{\infty } &=&\left\{ \left( \alpha _{1},\alpha _{2},\cdots \right) :\alpha
_{i}\geq 0\text{ for }1\leq i<\infty \text{, }\sum_{i=1}^{\infty }\alpha
_{i}=1,\right.  \label{eq3.14} \\
&&\left. \exists \xi _{1},\xi _{2},\cdots \in \mathbb{S}^{n}\text{ s.t. }%
\sum_{i=1}^{\infty }\alpha _{i}\xi _{i}=0\right\} .  \notag
\end{eqnarray}

\begin{lemma}
\label{lem3.2}For $0<\theta <1$ and $\left( \alpha _{1},\alpha _{2},\cdots
\right) \in W_{\infty }$, we have 
\begin{equation}
\sum_{i=1}^{\infty }\alpha _{i}^{\theta }\geq 2^{1-\theta }.  \label{eq3.15}
\end{equation}%
Equality holds if and only if $\left( \alpha _{1},\alpha _{2},\cdots \right)
=\left( \frac{1}{2},\frac{1}{2},0,\cdots \right) $ after permutation.
\end{lemma}

\begin{proof}
Let $\left( \alpha _{1},\alpha _{2},\cdots \right) \in W_{\infty }$, by
permutation we can assume $\alpha _{1}>0$. For $N\geq 2$, we can find $\xi
\in \mathbb{S}^{n}$ and $0\leq \beta \leq \sum_{i=N}^{\infty }\alpha _{i}$
s.t.%
\begin{equation}
\sum_{i=N}^{\infty }\alpha _{i}\xi _{i}=\beta \xi .  \label{eq3.16}
\end{equation}%
Hence%
\begin{equation}
\alpha _{1}\xi +\cdots +\alpha _{N-1}\xi _{N-1}+\beta \xi =0.  \label{eq3.17}
\end{equation}%
In particular,%
\begin{equation}
\frac{\left( \alpha _{1},\cdots ,\alpha _{N-1},\beta \right) }{\alpha
_{1}+\cdots +\alpha _{N-1}+\beta }\in W_{N}.  \label{eq3.18}
\end{equation}%
It follows from Lemma \ref{lem3.1} that%
\begin{equation}
\alpha _{1}^{\theta }+\cdots +\alpha _{N-1}^{\theta }+\beta ^{\theta }\geq
2^{1-\theta }\left( \alpha _{1}+\cdots +\alpha _{N-1}+\beta \right) ^{\theta
}.  \label{eq3.19}
\end{equation}%
Hence%
\begin{equation}
\alpha _{1}^{\theta }+\cdots +\alpha _{N-1}^{\theta }+\left(
\sum_{i=N}^{\infty }\alpha _{i}\right) ^{\theta }\geq 2^{1-\theta }\left(
\alpha _{1}+\cdots +\alpha _{N-1}\right) ^{\theta }.  \label{eq3.20}
\end{equation}%
Letting $N\rightarrow \infty $, we get $\sum_{i=1}^{\infty }\alpha
_{i}^{\theta }\geq 2^{1-\theta }$.

Next assume equality holds in (\ref{eq3.15}), we claim except for finitely
many $i$'s, $\alpha _{i}=0$. Once this is proved, it follows from Lemma \ref%
{lem3.1} that after permutation $\left( \alpha _{1},\alpha _{2},\cdots
\right) =\left( \frac{1}{2},\frac{1}{2},0,\cdots \right) $.

If for infinitely many $i$'s, $\alpha _{i}>0$, by removing all $0$'s, we can
assume $\alpha _{i}>0$ for all $i$. There exists $\xi _{1},\xi _{2},\cdots
\in \mathbb{S}^{n}$ s.t.%
\begin{equation}
\sum_{i=1}^{\infty }\alpha _{i}\xi _{i}=0.  \label{eq3.21}
\end{equation}%
Using $0<\theta <1$, by the same argument as in the proof of Lemma \ref%
{lem3.1}, we know for $i\neq j$, $\xi _{i}\neq \xi _{j}$.

Since $\alpha _{i}\rightarrow 0$ as $i\rightarrow \infty $, we see for $k$
large enough, $\alpha _{k}^{\theta -1}>2^{1-\theta }$. Now we can proceed in
the same way as the proof of Lemma \ref{lem3.1} to show for $i\neq k$, $\xi
_{i}=\pm \xi _{k}$. Indeed, without loss of generality, we assume $i=1$. For 
$\eta \in \mathbb{S}^{n}$ with $\left\langle \xi _{1},\eta \right\rangle $,
we let%
\begin{equation}
x\left( t\right) =-\alpha _{1}\left( \cos t\cdot \xi _{1}+\sin t\cdot \eta
\right) -\sum_{j>1,j\neq k}\alpha _{j}\xi _{j}.  \label{eq3.22}
\end{equation}%
Then $x\left( 0\right) =\alpha _{k}\xi _{k}\neq 0$. It follows that for $t$
small,%
\begin{equation*}
\left( \frac{\alpha _{1}}{1-\alpha _{k}+\left\vert x\right\vert },\cdots ,%
\frac{\alpha _{k-1}}{1-\alpha _{k}+\left\vert x\right\vert },\frac{%
\left\vert x\right\vert }{1-\alpha _{k}+\left\vert x\right\vert },\frac{%
\alpha _{k+1}}{1-\alpha _{k}+\left\vert x\right\vert },\cdots \right) \in
W_{\infty }.
\end{equation*}%
We have%
\begin{equation}
2^{1-\theta }\leq \frac{2^{1-\theta }-\alpha _{k}^{\theta }+\left\vert
x\right\vert ^{\theta }}{\left( 1-\alpha _{k}+\left\vert x\right\vert
\right) ^{\theta }}.  \label{eq3.23}
\end{equation}%
Hence%
\begin{equation*}
0=\left. \frac{d}{dt}\right\vert _{t=0}\frac{2^{1-\theta }-\alpha
_{k}^{\theta }+\left\vert x\right\vert ^{\theta }}{\left( 1-\alpha
_{k}+\left\vert x\right\vert \right) ^{\theta }}=\theta \alpha _{1}\left(
2^{1-\theta }-\alpha _{k}^{\theta -1}\right) \left\langle \xi _{k},\eta
\right\rangle .
\end{equation*}%
As a consequence $\left\langle \xi _{k},\eta \right\rangle =0$. It follows
that $\xi _{1}=\pm \xi _{k}$.

The fact $\xi _{i}=\pm \xi _{k}$ for $i\neq k$ contradicts with the fact $%
\xi _{j}$'s are mutually different.
\end{proof}

Now we are ready to derive Proposition \ref{prop3.1}.

\begin{proof}[Proof of Proposition \protect\ref{prop3.1}]
Let $\nu \in \mathcal{M}_{1}^{c}\left( \mathbb{S}^{n}\right) $. Assume $\nu $
is supported on $\left\{ \xi _{i}\right\} $ with $\nu _{i}=\nu \left(
\left\{ \xi _{i}\right\} \right) $, then $\nu _{i}\geq 0$, $\sum_{i}\nu
_{i}=1$ and $\sum_{i}\nu _{i}\xi _{i}=0$. Whether the support of $\nu $ is
finite or infinite, it follows from Lemma \ref{lem3.1} and \ref{lem3.2} that 
$\sum_{i}\nu _{i}^{\theta }\geq 2^{1-\theta }$. Hence $\Theta \left(
1,\theta ,n\right) \geq 2^{1-\theta }$. On the other hand, for any $\xi \in 
\mathbb{S}^{n}$, $\frac{1}{2}\delta _{\xi }+\frac{1}{2}\delta _{-\xi }\in 
\mathcal{M}_{1}^{c}\left( \mathbb{S}^{n}\right) $. Hence $\Theta \left(
1,\theta ,n\right) \leq 2^{1-\theta }$. In total we get $\Theta \left(
1,\theta ,n\right) =2^{1-\theta }$.

If $\Theta \left( 1,\theta ,n\right) $ is achieved at $\nu \in \mathcal{M}%
_{1}^{c}\left( \mathbb{S}^{n}\right) $, then it follows from Lemma \ref%
{lem3.1} and \ref{lem3.2} that $\nu =\frac{1}{2}\delta _{\xi }+\frac{1}{2}%
\delta _{-\xi }$ for some $\xi \in \mathbb{S}^{n}$.
\end{proof}

\begin{proposition}
\label{prop3.2}For $\theta \in \left( 0,1\right) $ and $n\in \mathbb{N}$,%
\begin{equation}
\Theta \left( 2,\theta ,n\right) =\left( n+2\right) ^{1-\theta }.
\label{eq3.24}
\end{equation}%
Moreover $\Theta \left( 2,\theta ,n\right) $ is achieved at $\nu \in 
\mathcal{M}_{2}^{c}\left( \mathbb{S}^{n}\right) $ if and only if 
\begin{equation}
\nu =\frac{1}{n+2}\sum_{i=1}^{n+2}\delta _{\xi _{i}}  \label{eq3.25}
\end{equation}%
for $\xi _{1},\cdots ,\xi _{n+2}\in \mathbb{S}^{n}$ being the vortices of a
regular $\left( n+1\right) $-simplex embedded in the unit ball.
\end{proposition}

Before starting the proof of Proposition \ref{prop3.2}, we make some
observations about $\mathcal{M}_{2}^{c}\left( \mathbb{S}^{n}\right) $.
Recall (see \cite[chapter IV]{SW})%
\begin{equation}
\dim \left( \left. \overset{\circ }{\mathcal{P}}_{2}\right\vert _{\mathbb{S}%
^{n}}\right) =\frac{n^{2}+3n}{2}+n+1.  \label{eq3.26}
\end{equation}%
Moreover $\left. \overset{\circ }{\mathcal{P}}_{2}\right\vert _{\mathbb{S}%
^{n}}$ has a base%
\begin{equation}
x_{1},\cdots ,x_{n+1},x_{1}^{2}-\frac{\left\vert x\right\vert ^{2}}{n+1}%
,\cdots ,x_{n}^{2}-\frac{\left\vert x\right\vert ^{2}}{n+1},x_{i}x_{j}\text{
for }1\leq i<j\leq n+1.  \label{eq3.27}
\end{equation}

Let $\nu $ be a probability measure on $\mathbb{S}^{n}$. Assume $\nu $ is
supported on $N$ points $\left\{ \xi _{i}\right\} _{i=1}^{N}\subset \mathbb{S%
}^{n}$. Denote $\nu _{i}=\nu \left( \left\{ \xi _{i}\right\} \right) $. We
define $n+2$ vectors in $\mathbb{R}^{N}$ as%
\begin{equation}
u_{0}=\left[ 
\begin{array}{c}
\sqrt{\nu _{1}} \\ 
\sqrt{\nu _{2}} \\ 
\vdots \\ 
\sqrt{\nu _{N}}%
\end{array}%
\right] ,u_{j}=\left[ 
\begin{array}{c}
\sqrt{\left( n+1\right) \nu _{1}}\xi _{1,j} \\ 
\sqrt{\left( n+1\right) \nu _{2}}\xi _{2,j} \\ 
\vdots \\ 
\sqrt{\left( n+1\right) \nu _{N}}\xi _{N,j}%
\end{array}%
\right] \text{ for }1\leq j\leq n+1.  \label{eq3.28}
\end{equation}%
Here $\xi _{i,j}$ is the $j$th coordinate of $\xi _{i}$ as a vector in $%
\mathbb{R}^{n+1}$.

\begin{lemma}
\label{lem3.3}Let $\nu $ be a probability measure on $\mathbb{S}^{n}$
supported on $N$ points. Then%
\begin{equation}
\nu \in \mathcal{M}_{2}^{c}\left( \mathbb{S}^{n}\right) \Leftrightarrow
u_{0},u_{1},\cdots ,u_{n+1}\text{ is orthonormal in }\mathbb{R}^{N}.
\label{eq3.29}
\end{equation}%
Here $u_{0},u_{1},\cdots ,u_{n+1}$ is defined in (\ref{eq3.28}) and $\mathbb{%
R}^{N}$ is equipped with standard inner product.
\end{lemma}

\begin{proof}
This follows from (\ref{eq3.27}).
\end{proof}

Similarly, if $\nu $ is supported on countable infinitely many points $%
\left\{ \xi _{i}\right\} _{i=1}^{\infty }\subset \mathbb{S}^{n}$ with $\nu
_{i}=\nu \left( \left\{ \xi _{i}\right\} \right) $, we can define $n+2$
vectors in the Hilbert space $\ell ^{2}$ as%
\begin{equation}
\begin{array}{l}
u_{0}=\left( \sqrt{\nu _{1}},\sqrt{\nu _{2}},\cdots \right) , \\ 
u_{1}=\left( \sqrt{\left( n+1\right) \nu _{1}}\xi _{1,1},\sqrt{\left(
n+1\right) \nu _{2}}\xi _{2,1},\cdots \right) , \\ 
\cdots \\ 
u_{n+1}=\left( \sqrt{\left( n+1\right) \nu _{1}}\xi _{1,n+1},\sqrt{\left(
n+1\right) \nu _{2}}\xi _{2,n+1},\cdots \right) .%
\end{array}
\label{eq3.30}
\end{equation}%
Here%
\begin{equation}
\ell ^{2}=\left\{ \left( c_{1},c_{2},\cdots \right) :c_{i}\in \mathbb{R}%
,\sum_{i=1}^{\infty }c_{i}^{2}<\infty \right\}  \label{eq3.31}
\end{equation}%
and it is equipped with the standard inner product.

\begin{lemma}
\label{lem3.4}Let $\nu $ be a probability measure on $\mathbb{S}^{n}$
supported on countable infinitely many points. Then%
\begin{equation}
\nu \in \mathcal{M}_{2}^{c}\left( \mathbb{S}^{n}\right) \Leftrightarrow
u_{0},u_{1},\cdots ,u_{n+1}\text{ is orthonormal in }\ell ^{2}.
\label{eq3.32}
\end{equation}%
Here $u_{0},u_{1},\cdots ,u_{n+1}$ is defined in (\ref{eq3.30}).
\end{lemma}

Again this lemma follows from (\ref{eq3.27}).

\begin{proof}[Proof of Proposition \protect\ref{prop3.2}]
Let $\nu \in \mathcal{M}_{2}^{c}\left( \mathbb{S}^{n}\right) $. If $\nu $ is
supported on finitely many points, say $\left\{ \xi _{i}\right\}
_{i=1}^{N}\subset \mathbb{S}^{n}$, then by Lemma \ref{lem3.3} we know $%
u_{0},\cdots ,u_{n+1}$ is orthonormal in $\mathbb{R}^{N}$. Let $e_{1},\cdots
,e_{N}$ be the standard base of $\mathbb{R}^{N}$, then for $1\leq i\leq N$,%
\begin{equation}
\sum_{j=0}^{n+1}\left\langle e_{i},u_{j}\right\rangle ^{2}=\left( n+2\right)
\nu _{i}.  \label{eq3.33}
\end{equation}%
Here $\left\langle \cdot ,\cdot \right\rangle $ is the standard inner
product on $\mathbb{R}^{N}$. It follows from Parseval's relation that%
\begin{equation}
0<\sum_{j=0}^{n+1}\left\langle e_{i},u_{j}\right\rangle ^{2}\leq 1.
\label{eq3.34}
\end{equation}%
Hence using $0<\theta <1$, we have%
\begin{eqnarray*}
\sum_{i=1}^{N}\nu _{i}^{\theta } &=&\frac{1}{\left( n+2\right) ^{\theta }}%
\sum_{i=1}^{N}\left( \sum_{j=0}^{n+1}\left\langle e_{i},u_{j}\right\rangle
^{2}\right) ^{\theta } \\
&\geq &\frac{1}{\left( n+2\right) ^{\theta }}\sum_{i=1}^{N}\sum_{j=0}^{n+1}%
\left\langle e_{i},u_{j}\right\rangle ^{2} \\
&=&\frac{1}{\left( n+2\right) ^{\theta }}\sum_{j=0}^{n+1}\sum_{i=1}^{N}\left%
\langle e_{i},u_{j}\right\rangle ^{2} \\
&=&\frac{1}{\left( n+2\right) ^{\theta }}\sum_{j=0}^{n+1}1 \\
&=&\left( n+2\right) ^{1-\theta }.
\end{eqnarray*}%
If $\sum_{i=1}^{N}\nu _{i}^{\theta }=\left( n+2\right) ^{1-\theta }$, then
for $1\leq i\leq N$,%
\begin{equation*}
\sum_{j=0}^{n+1}\left\langle e_{i},u_{j}\right\rangle ^{2}=1.
\end{equation*}%
In view of (\ref{eq3.33}) we know $\nu _{i}=\frac{1}{n+2}$. Using $%
\sum_{i=1}^{N}\nu _{i}=1$, we see $N=n+2$. Moreover%
\begin{equation*}
e_{i}=\sum_{j=1}^{n+1}\left\langle e_{i},u_{j}\right\rangle u_{j}\in 
\limfunc{span}\left\{ u_{0},u_{1},\cdots ,u_{n+1}\right\} .
\end{equation*}%
It follows that $u_{0},u_{1},\cdots ,u_{n+1}$ is an orthonormal base for $%
\mathbb{R}^{n+2}$. This implies the matrix%
\begin{equation*}
A=\left[ u_{0},u_{1},\cdots ,u_{n+1}\right]
\end{equation*}%
is orthonormal. Since%
\begin{equation*}
A=\left[ 
\begin{array}{cc}
\frac{1}{\sqrt{n+2}} & \sqrt{\frac{n+1}{n+2}}\xi _{1}^{T} \\ 
\frac{1}{\sqrt{n+2}} & \sqrt{\frac{n+1}{n+2}}\xi _{2}^{T} \\ 
\vdots & \vdots \\ 
\frac{1}{\sqrt{n+2}} & \sqrt{\frac{n+1}{n+2}}\xi _{n+2}^{T}%
\end{array}%
\right] ,
\end{equation*}%
we know for $1\leq i<j\leq n+2$, $\left\Vert \xi _{i}-\xi _{j}\right\Vert =%
\sqrt{\frac{2\left( n+2\right) }{n+1}}$. Hence $\nu =\frac{1}{n+2}%
\sum_{i=1}^{n+2}\delta _{\xi _{i}}$ and $\xi _{1},\cdots ,\xi _{n+2}\in 
\mathbb{S}^{n}$ are the vortices of a regular $\left( n+1\right) $-simplex
embedded in the unit ball.

If $\nu $ is supported on countable infinitely many points, say $\left\{ \xi
_{i}\right\} _{i=1}^{\infty }\subset \mathbb{S}^{n}$, then by Lemma \ref%
{lem3.4} we know $u_{0},\cdots ,u_{n+1}$ is orthonormal in $\mathbb{\ell }%
^{2}$. Let%
\begin{equation}
e_{i}=\left( 0,\cdots ,0,1,0,\cdots \right) ,\text{ }1\text{ lies at }i\text{%
th position.}  \label{eq3.35}
\end{equation}%
For $x,y\in \ell ^{2}$, let%
\begin{equation}
\left\langle x,y\right\rangle =\dsum\limits_{i=1}^{\infty }x_{i}y_{i}
\label{eq3.36}
\end{equation}%
be the standard inner product. Then for all $i$,%
\begin{equation}
\sum_{j=0}^{n+1}\left\langle e_{i},u_{j}\right\rangle ^{2}=\left( n+2\right)
\nu _{i}.  \label{eq3.37}
\end{equation}%
By Parseval's relation in $\mathbb{\ell }^{2}$, we know%
\begin{equation}
0<\sum_{j=0}^{n+1}\left\langle e_{i},u_{j}\right\rangle ^{2}\leq 1.
\label{eq3.38}
\end{equation}%
As in the finite support points case, we have%
\begin{eqnarray*}
\sum_{i=1}^{\infty }\nu _{i}^{\theta } &=&\frac{1}{\left( n+2\right)
^{\theta }}\sum_{i=1}^{\infty }\left( \sum_{j=0}^{n+1}\left\langle
e_{i},u_{j}\right\rangle ^{2}\right) ^{\theta } \\
&\geq &\frac{1}{\left( n+2\right) ^{\theta }}\sum_{i=1}^{\infty
}\sum_{j=0}^{n+1}\left\langle e_{i},u_{j}\right\rangle ^{2} \\
&=&\frac{1}{\left( n+2\right) ^{\theta }}\sum_{j=0}^{n+1}\sum_{i=1}^{\infty
}\left\langle e_{i},u_{j}\right\rangle ^{2} \\
&=&\left( n+2\right) ^{1-\theta }.
\end{eqnarray*}%
In this case, we always have%
\begin{equation}
\sum_{i=1}^{\infty }\nu _{i}^{\theta }>\left( n+2\right) ^{1-\theta }.
\label{eq3.39}
\end{equation}%
Indeed if $\sum_{i=1}^{\infty }\nu _{i}^{\theta }=\left( n+2\right)
^{1-\theta }$, then for all $i$,%
\begin{equation}
\sum_{j=0}^{n+1}\left\langle e_{i},u_{j}\right\rangle ^{2}=1.  \label{eq3.40}
\end{equation}%
By (\ref{eq3.37}) we see $\nu _{i}=\frac{1}{n+2}$. This contradicts with the
fact $\sum_{i=1}^{\infty }\nu _{i}=1$.

Summing up we see $\Theta \left( 2,\theta ,n\right) \geq \left( n+2\right)
^{1-\theta }$. On the other hand, let $\xi _{1},\cdots ,\xi _{n+2}\in 
\mathbb{S}^{n}$ be the vortices of a regular $\left( n+1\right) $-simplex
embedded in the unit ball, then it follows from \cite[proof of Lemma 3.1]{Ha}
that $\frac{1}{n+2}\delta _{\xi _{1}}+\cdots +\frac{1}{n+2}\delta _{\xi
_{n+2}}\in \mathcal{M}_{2}^{c}\left( \mathbb{S}^{n}\right) $. Hence $\Theta
\left( 2,\theta ,n\right) \leq \left( n+2\right) ^{1-\theta }$. It follows
that $\Theta \left( 2,\theta ,n\right) =\left( n+2\right) ^{1-\theta }$.
\end{proof}

It is interesting that similar idea as in the above proof can help us
finding $\Theta \left( m,\theta ,1\right) $ for all $m\in \mathbb{N}$.

\begin{proposition}
\label{prop3.3}For $m\in \mathbb{N}$ and $\theta \in \left( 0,1\right) $,%
\begin{equation}
\Theta \left( m,\theta ,1\right) =\left( m+1\right) ^{1-\theta }.
\label{eq3.41}
\end{equation}%
Moreover $\Theta \left( m,\theta ,1\right) $ is achieved at $\nu \in 
\mathcal{M}_{m}^{c}\left( \mathbb{S}^{1}\right) $ if and only if%
\begin{equation}
\nu =\frac{1}{m+1}\sum_{j=0}^{m}\delta _{e^{i\left( \alpha +\frac{2j\pi }{m+1%
}\right) }}  \label{eq3.42}
\end{equation}%
for some $\alpha \in \mathbb{R}$.
\end{proposition}

Let $\nu $ be a probability measure on $\mathbb{S}^{1}$. Assume $\nu $ is
supported on $N$ points $\left\{ \xi _{j}\right\} _{j=1}^{N}\subset \mathbb{S%
}^{1}$ with $\nu _{j}=\nu \left( \left\{ \xi _{j}\right\} \right) $. Let%
\begin{equation}
u_{0}=\left[ 
\begin{array}{c}
\sqrt{\nu _{1}} \\ 
\sqrt{\nu _{2}} \\ 
\vdots \\ 
\sqrt{\nu _{N}}%
\end{array}%
\right] ,u_{1}=\left[ 
\begin{array}{c}
\sqrt{\nu _{1}}\xi _{1} \\ 
\sqrt{\nu _{2}}\xi _{2} \\ 
\vdots \\ 
\sqrt{\nu _{N}}\xi _{N}%
\end{array}%
\right] ,\cdots ,u_{m}=\left[ 
\begin{array}{c}
\sqrt{\nu _{1}}\xi _{1}^{m} \\ 
\sqrt{\nu _{2}}\xi _{2}^{m} \\ 
\vdots \\ 
\sqrt{\nu _{N}}\xi _{N}^{m}%
\end{array}%
\right] ,  \label{eq3.43}
\end{equation}%
then it is clear that%
\begin{equation}
\nu \in \mathcal{M}_{m}^{c}\left( \mathbb{S}^{1}\right) \Leftrightarrow
u_{0},\cdots ,u_{m}\text{ is orthonormal in }\mathbb{C}^{N}.  \label{eq3.44}
\end{equation}%
Here $\mathbb{C}^{N}$ is equipped with the standard Hermite inner product
i.e. $\left\langle x,y\right\rangle =x^{T}\overline{y}$ for $x,y\in \mathbb{C%
}^{N}$.

Assume $\nu $ is supported on countable infinitely many points $\left\{ \xi
_{j}\right\} _{j=1}^{\infty }\subset \mathbb{S}^{1}$ with $\nu _{j}=\nu
\left( \left\{ \xi _{j}\right\} \right) $. We let%
\begin{equation}
\ell _{\mathbb{C}}^{2}=\left\{ \left( x_{j}\right) _{j=1}^{\infty }:x_{j}\in 
\mathbb{C},\sum_{j=1}^{\infty }\left\vert x_{j}\right\vert ^{2}<\infty
\right\} .  \label{eq3.45}
\end{equation}%
It is equipped with the Hermite inner product%
\begin{equation}
\left\langle x,y\right\rangle =\sum_{j=1}^{\infty }x_{j}\overline{y_{j}}%
\text{ for }x,y\in \ell _{\mathbb{C}}^{2}.  \label{eq3.46}
\end{equation}%
We construct $m+1$ vectors in $\ell _{\mathbb{C}}^{2}$,%
\begin{equation}
\begin{array}{l}
u_{0}=\left( \sqrt{\nu _{1}},\sqrt{\nu _{2}},\cdots \right) , \\ 
u_{1}=\left( \sqrt{\nu _{1}}\xi _{1},\sqrt{\nu _{2}}\xi _{2},\cdots \right) ,
\\ 
\cdots \\ 
u_{m}=\left( \sqrt{\nu _{1}}\xi _{1}^{m},\sqrt{\nu _{2}}\xi _{2}^{m},\cdots
\right) .%
\end{array}
\label{eq3.47}
\end{equation}%
Then again we have%
\begin{equation}
\nu \in \mathcal{M}_{m}^{c}\left( \mathbb{S}^{1}\right) \Leftrightarrow
u_{0},\cdots ,u_{m}\text{ is orthonormal in }\ell _{\mathbb{C}}^{2}.
\label{eq3.48}
\end{equation}

\begin{proof}[Proof of Proposition \protect\ref{prop3.3}]
Let $\nu \in \mathcal{M}_{m}^{c}\left( \mathbb{S}^{1}\right) $. If $\nu $ is
supported on finitely many points, namely $\left\{ \xi _{j}\right\}
_{j=1}^{N}\subset \mathbb{S}^{1}$, then $u_{0},\cdots ,u_{m}$ defined in (%
\ref{eq3.43}) is orthonormal in $\mathbb{C}^{N}$, Let $e_{1},\cdots ,e_{N}$
be the standard base of $\mathbb{C}^{N}$, then for $1\leq j\leq N$, we have%
\begin{equation}
\sum_{k=0}^{m}\left\vert \left\langle e_{j},u_{k}\right\rangle \right\vert
^{2}=\left( m+1\right) \nu _{j}.  \label{eq3.49}
\end{equation}%
Note that%
\begin{equation}
0<\sum_{k=0}^{m}\left\vert \left\langle e_{j},u_{k}\right\rangle \right\vert
^{2}\leq 1.  \label{eq3.50}
\end{equation}%
Hence%
\begin{eqnarray*}
\sum_{j=1}^{N}\nu _{j}^{\theta } &=&\frac{1}{\left( m+1\right) ^{\theta }}%
\sum_{j=1}^{N}\left( \sum_{k=0}^{m}\left\vert \left\langle
e_{j},u_{k}\right\rangle \right\vert ^{2}\right) ^{\theta } \\
&\geq &\frac{1}{\left( m+1\right) ^{\theta }}\sum_{j=1}^{N}\sum_{k=0}^{m}%
\left\vert \left\langle e_{j},u_{k}\right\rangle \right\vert ^{2} \\
&=&\frac{1}{\left( m+1\right) ^{\theta }}\sum_{k=0}^{m}\sum_{j=1}^{N}\left%
\vert \left\langle e_{j},u_{k}\right\rangle \right\vert ^{2} \\
&=&\frac{1}{\left( m+1\right) ^{\theta }}\sum_{k=0}^{m}1 \\
&=&\left( m+1\right) ^{1-\theta }.
\end{eqnarray*}%
If $\sum_{j=1}^{N}\nu _{j}^{\theta }=\left( m+1\right) ^{1-\theta }$, then
for $1\leq j\leq N$,%
\begin{equation}
\sum_{k=0}^{m}\left\vert \left\langle e_{j},u_{k}\right\rangle \right\vert
^{2}=1.  \label{eq3.51}
\end{equation}%
By (\ref{eq3.49}) we know $\nu _{j}=\frac{1}{m+1}$. Since $\sum_{j=1}^{N}\nu
_{j}=1$, we see $N=m+1$. Moreover%
\begin{equation}
e_{j}=\sum_{k=0}^{m}\left\langle e_{j},u_{k}\right\rangle u_{k}\in \limfunc{%
span}\left\{ u_{0},\cdots ,u_{m}\right\} .  \label{eq3.52}
\end{equation}%
Hence $u_{0},\cdots ,u_{m}$ is an orthonormal base for $\mathbb{C}^{m+1}$.
This implies the matrix%
\begin{equation}
A=\left[ u_{0},u_{1},\cdots ,u_{m}\right] =\left[ 
\begin{array}{cccc}
\frac{1}{\sqrt{m+1}} & \frac{\xi _{1}}{\sqrt{m+1}} & \cdots & \frac{\xi
_{1}^{m}}{\sqrt{m+1}} \\ 
\frac{1}{\sqrt{m+1}} & \frac{\xi _{2}}{\sqrt{m+1}} & \cdots & \frac{\xi
_{2}^{m}}{\sqrt{m+1}} \\ 
\vdots & \vdots & \ddots & \vdots \\ 
\frac{1}{\sqrt{m+1}} & \frac{\xi _{m+1}}{\sqrt{m+1}} & \cdots & \frac{\xi
_{m+1}^{m}}{\sqrt{m+1}}%
\end{array}%
\right]  \label{eq3.53}
\end{equation}%
is unitary. Let $z_{j}=\frac{\xi _{j}}{\xi _{m+1}}=\xi _{j}\overline{\xi
_{m+1}}$ for $1\leq j\leq m$, then%
\begin{equation}
1+z_{j}+\cdots +z_{j}^{m}=0.  \label{eq3.54}
\end{equation}%
Hence $z_{j}^{m+1}=1$. Since $z_{j}\neq 1$ and $z_{j}\neq z_{k}$ for $j\neq
k $, we see%
\begin{equation}
\left\{ z_{1},\cdots ,z_{m}\right\} =\left\{ e^{i\frac{2\pi }{m+1}},e^{i%
\frac{4\pi }{m+1}},\cdots ,e^{i\frac{2m\pi }{m+1}}\right\} .  \label{eq3.55}
\end{equation}%
Since we can write $\xi _{m+1}=e^{i\alpha }$ for some $\alpha \in \mathbb{R}$%
, we see%
\begin{equation}
\left\{ \xi _{1},\cdots ,\xi _{m+1}\right\} =\left\{ e^{i\alpha },e^{i\left(
\alpha +\frac{2\pi }{m+1}\right) },\cdots ,e^{i\left( \alpha +\frac{2m\pi }{%
m+1}\right) }\right\} .  \label{eq3.56}
\end{equation}%
Hence%
\begin{equation}
\nu =\frac{1}{m+1}\sum_{j=0}^{m}\delta _{e^{i\left( \alpha +\frac{2j\pi }{m+1%
}\right) }}.  \label{eq3.57}
\end{equation}

Assume $\nu $ is supported on countable infinitely many points, namely $%
\left\{ \xi _{j}\right\} _{j=1}^{\infty }\subset \mathbb{S}^{1}$ with $\nu
_{j}=\nu \left( \left\{ \xi _{j}\right\} \right) $, then $u_{0},\cdots
,u_{m} $ given by (\ref{eq3.47}) is orthonormal in $\mathbb{\ell }_{\mathbb{C%
}}^{2}$. Let%
\begin{equation}
e_{j}=\left( 0,\cdots ,0,1,0,\cdots \right) ,\text{ }1\text{ lies at }j\text{%
th position.}  \label{eq3.58}
\end{equation}%
Then for all $j$,%
\begin{equation}
\sum_{k=0}^{m}\left\vert \left\langle e_{j},u_{k}\right\rangle \right\vert
^{2}=\left( m+1\right) \nu _{j}.  \label{eq3.59}
\end{equation}%
By Parseval's relation in $\mathbb{\ell }_{\mathbb{C}}^{2}$, we have%
\begin{equation}
0<\sum_{k=0}^{m}\left\vert \left\langle e_{j},u_{k}\right\rangle \right\vert
^{2}\leq 1.  \label{eq3.60}
\end{equation}%
Hence%
\begin{eqnarray*}
\sum_{j=1}^{\infty }\nu _{j}^{\theta } &=&\frac{1}{\left( m+1\right)
^{\theta }}\sum_{j=1}^{\infty }\left( \sum_{k=0}^{m}\left\vert \left\langle
e_{j},u_{k}\right\rangle \right\vert ^{2}\right) ^{\theta } \\
&\geq &\frac{1}{\left( m+1\right) ^{\theta }}\sum_{j=1}^{\infty
}\sum_{k=0}^{m}\left\vert \left\langle e_{j},u_{k}\right\rangle \right\vert
^{2} \\
&=&\frac{1}{\left( m+1\right) ^{\theta }}\sum_{k=0}^{m}\sum_{j=1}^{\infty
}\left\vert \left\langle e_{j},u_{k}\right\rangle \right\vert ^{2} \\
&=&\left( m+1\right) ^{1-\theta }.
\end{eqnarray*}%
Here we always have%
\begin{equation}
\sum_{j=1}^{\infty }\nu _{j}^{\theta }>\left( m+1\right) ^{1-\theta }.
\label{eq3.61}
\end{equation}%
Indeed if $\sum_{j=1}^{\infty }\nu _{j}^{\theta }=\left( m+1\right)
^{1-\theta }$, then for all $j$,%
\begin{equation}
\sum_{k=0}^{m}\left\vert \left\langle e_{j},u_{k}\right\rangle \right\vert
^{2}=1.  \label{eq3.62}
\end{equation}%
By (\ref{eq3.59}) we see $\nu _{j}=\frac{1}{m+1}$. This contradicts with the
fact $\sum_{j=1}^{\infty }\nu _{j}=1$.

Sum up the two cases, we see $\Theta \left( m,\theta ,1\right) \geq \left(
m+1\right) ^{1-\theta }$. On the other hand, for any $\alpha \in \mathbb{R}$%
, it is clear that%
\begin{equation}
\frac{1}{m+1}\sum_{j=0}^{m}\delta _{e^{i\left( \alpha +\frac{2j\pi }{m+1}%
\right) }}\in \mathcal{M}_{m}^{c}\left( \mathbb{S}^{1}\right) .
\label{eq3.63}
\end{equation}%
Hence $\Theta \left( m,\theta ,1\right) \leq \left( m+1\right) ^{1-\theta }$%
. It follows that $\Theta \left( m,\theta ,1\right) =\left( m+1\right)
^{1-\theta }$.
\end{proof}

\section{Further discussions\label{sec4}}

In this section we will first show the constant $\frac{S_{n,p}^{p}}{\left(
n+2\right) ^{\frac{p}{n}}}+\varepsilon $ in Corollary \ref{cor1.1} is almost
optimal. Then we will discuss the application of our approach to higher
order Sobolev spaces and its close relation to the fourth order $Q$
curvature equations in conformal geometry.

\begin{lemma}
\label{lem4.1}Assume $n\in \mathbb{N}$ and $1<p<n$. Let $p^{\ast }=\frac{pn}{%
n-p}$. If $a,b\in \mathbb{R}$ s.t. for any $u\in W^{1,p}\left( \mathbb{S}%
^{n}\right) $ with%
\begin{equation}
\int_{\mathbb{S}^{n}}f\left\vert u\right\vert ^{p^{\ast }}d\mu =0
\label{eq4.1}
\end{equation}%
for all $f\in \overset{\circ }{\mathcal{P}}_{2}$, we have%
\begin{equation}
\left\Vert u\right\Vert _{L^{p^{\ast }}\left( \mathbb{S}^{n}\right)
}^{p}\leq a\left\Vert \nabla u\right\Vert _{L^{p}\left( \mathbb{S}%
^{n}\right) }^{p}+b\left\Vert u\right\Vert _{L^{p}\left( \mathbb{S}%
^{n}\right) }^{p},  \label{eq4.2}
\end{equation}%
then%
\begin{equation}
a\geq \frac{S_{n,p}^{p}}{\left( n+2\right) ^{\frac{p}{n}}}.  \label{eq4.3}
\end{equation}%
Here $S_{n,p}$ is given by 
\begin{equation}
S_{n,p}=\frac{1}{n}\left( \frac{n\left( p-1\right) }{n-p}\right) ^{1-\frac{1%
}{p}}\left( \frac{n!}{\Gamma \left( \frac{n}{p}\right) \Gamma \left( n+1-%
\frac{n}{p}\right) \left\vert \mathbb{S}^{n-1}\right\vert }\right) ^{\frac{1%
}{n}}  \label{eq4.4}
\end{equation}
\end{lemma}

Our approach to Lemma \ref{lem4.1} is in the same spirit as the method in
the proof of \cite[Lemma \ref{lem3.1}]{ChH}. We first recall for $\alpha
,\beta >0$,%
\begin{equation}
B\left( \alpha ,\beta \right) =\int_{0}^{1}\left( 1-t\right) ^{\alpha
-1}t^{\beta -1}dt  \label{eq4.5}
\end{equation}%
and%
\begin{equation}
B\left( \alpha ,\beta \right) =\frac{\Gamma \left( \alpha \right) \Gamma
\left( \beta \right) }{\Gamma \left( \alpha +\beta \right) }.  \label{eq4.6}
\end{equation}

Let $x_{1},\cdots ,x_{n+2}\in \mathbb{S}^{n}$ be the vertices of a regular $%
\left( n+1\right) $-simplex embedded in the unit ball, then by \cite[proof
of Lemma 3.1]{Ha} we know $\frac{1}{n+2}\sum_{i=1}^{n+2}\delta _{x_{i}}\in 
\mathcal{M}_{2}^{c}\left( \mathbb{S}^{n}\right) $. For $x,y\in \mathbb{S}%
^{n} $, we denote $\overline{xy}$ as the geodesic distance between $x$ and $%
y $ on $\mathbb{S}^{n}$. For $r>0$ and $x\in \mathbb{S}^{n}$, we define%
\begin{equation}
B_{r}\left( x\right) =\left\{ y\in \mathbb{S}^{n}:\overline{xy}<r\right\} .
\label{eq4.7}
\end{equation}

Let $\delta >0$ be small enough such that for $1\leq i<j\leq n+2$, $%
\overline{B_{2\delta }\left( x_{i}\right) }\cap \overline{B_{2\delta }\left(
x_{j}\right) }=\emptyset $. For $0<\varepsilon <\delta $, we let%
\begin{equation}
\phi _{\varepsilon }\left( t\right) =\left\{ 
\begin{array}{ll}
\left( \varepsilon +t^{p^{\prime }}\right) ^{-\frac{n}{p\ast }}, & 
0<t<\delta , \\ 
\left( \varepsilon +\delta ^{p^{\prime }}\right) ^{-\frac{n}{p\ast }}\left(
2-\frac{t}{\delta }\right) , & \delta <t<2\delta , \\ 
0 & t>2\delta .%
\end{array}%
\right.  \label{eq4.8}
\end{equation}%
Here%
\begin{equation}
p^{\prime }=\frac{p}{p-1},\quad p^{\ast }=\frac{pn}{n-p}.  \label{eq4.9}
\end{equation}%
Define%
\begin{equation}
v\left( x\right) =\sum_{i=1}^{n+2}\phi _{\varepsilon }\left( \overline{xx_{i}%
}\right) ,  \label{eq4.10}
\end{equation}%
then%
\begin{eqnarray}
&&\int_{\mathbb{S}^{n}}v^{p^{\ast }}d\mu  \label{eq4.11} \\
&=&\sum_{i=1}^{n+2}\int_{B_{2\delta }\left( x_{i}\right) }\phi _{\varepsilon
}\left( \overline{xx_{i}}\right) ^{p^{\ast }}d\mu \left( x\right)  \notag \\
&=&\left( n+2\right) \left\vert \mathbb{S}^{n-1}\right\vert
\int_{0}^{2\delta }\phi _{\varepsilon }\left( r\right) ^{p^{\ast }}\sin
^{n-1}rdr  \notag \\
&=&\varepsilon ^{-\frac{n}{p}}\left\vert \mathbb{S}^{n-1}\right\vert \frac{%
n+2}{p^{\prime }}B\left( \frac{n}{p},\frac{n}{p^{\prime }}\right) +O\left(
\varepsilon ^{-\frac{n}{p}+\tau }\right)  \notag
\end{eqnarray}%
as $\varepsilon \rightarrow 0^{+}$. Here $\tau $ is a fixed number satisfying%
\begin{equation}
0<\tau <\min \left\{ \frac{n}{p},\frac{2}{p^{\prime }},\frac{p^{\ast }}{%
p^{\prime }}\right\} .  \label{eq4.12}
\end{equation}%
For any $f\in \overset{\circ }{\mathcal{P}}_{2}$, we have%
\begin{eqnarray}
&&\int_{\mathbb{S}^{n}}v^{p^{\ast }}fd\mu  \label{eq4.13} \\
&=&\sum_{i=1}^{n+2}\int_{B_{2\delta }\left( x_{i}\right) }\phi _{\varepsilon
}\left( \overline{xx_{i}}\right) ^{p^{\ast }}f\left( x\right) d\mu \left(
x\right)  \notag \\
&=&\sum_{i=1}^{n+2}\left( \int_{B_{2\delta }\left( x_{i}\right) }\phi
_{\varepsilon }\left( \overline{xx_{i}}\right) ^{p^{\ast }}f\left(
x_{i}\right) d\mu \left( x\right) +\int_{B_{2\delta }\left( x_{i}\right)
}\phi _{\varepsilon }\left( \overline{xx_{i}}\right) ^{p^{\ast }}O\left( 
\overline{xx_{i}}^{2}\right) d\mu \left( x\right) \right)  \notag \\
&=&O\left( \varepsilon ^{-\frac{n}{p}+\tau }\right)  \notag
\end{eqnarray}%
as $\varepsilon \rightarrow 0^{+}$. Here we have used the equality%
\begin{equation}
\sum_{i=1}^{n+2}f\left( x_{i}\right) =0.  \label{eq4.14}
\end{equation}

To get a test function satisfying (\ref{eq4.1}), we need to do some
corrections. Let us fix a base of $\left. \overset{\circ }{\mathcal{P}}%
_{2}\right\vert _{\mathbb{S}^{n}}$, namely $\left. f_{1}\right\vert _{%
\mathbb{S}^{n}},\cdots ,\left. f_{l}\right\vert _{\mathbb{S}^{n}}$, here $%
f_{1},\cdots ,f_{l}\in \overset{\circ }{\mathcal{P}}_{2}$ and%
\begin{equation}
l=\frac{n^{2}+3n}{2}+n+1.  \label{eq4.15}
\end{equation}%
We can find $\psi _{1},\cdots ,\psi _{l}\in C_{c}^{\infty }\left( \mathbb{S}%
^{n}\backslash \dbigcup\limits_{i=1}^{N}\overline{B_{2\delta }\left(
x_{i}\right) }\right) $ such that the determinant%
\begin{equation}
\det \left[ \int_{\mathbb{S}^{n}}\psi _{j}f_{k}d\mu \right] _{1\leq j,k\leq
l}\neq 0.  \label{eq4.16}
\end{equation}%
Indeed, fix a nonzero smooth function $\eta \in C_{c}^{\infty }\left( 
\mathbb{S}^{n}\backslash \dbigcup\limits_{i=1}^{N}\overline{B_{2\delta
}\left( x_{i}\right) }\right) $, then $\eta f_{1},\cdots ,\eta f_{l}$ are
linearly independent. It follows that the Gram matrix%
\begin{equation*}
\left[ \int_{\mathbb{S}^{n}}\eta ^{2}f_{j}f_{k}d\mu \right] _{1\leq j,k\leq
l}
\end{equation*}%
is positive definite. Then $\psi _{j}=\eta ^{2}p_{j}$ satisfies (\ref{eq4.16}%
).

It follows from (\ref{eq4.16}) that we can find $\beta _{1},\cdots ,\beta
_{l}\in \mathbb{R}$ such that%
\begin{equation}
\int_{\mathbb{S}^{n}}\left( v^{p^{\ast }}+\sum_{j=1}^{l}\beta _{j}\psi
_{j}\right) f_{k}d\mu =0  \label{eq4.17}
\end{equation}%
for $k=1,\cdots ,l$. Moreover by (\ref{eq4.13}) we have%
\begin{equation}
\beta _{j}=O\left( \varepsilon ^{-\frac{n}{p}+\tau }\right)  \label{eq4.18}
\end{equation}%
as $\varepsilon \rightarrow 0^{+}$. As a consequence we can find a constant $%
c_{1}>0$ such that%
\begin{equation}
\sum_{j=1}^{l}\beta _{j}\psi _{j}+c_{1}\varepsilon ^{-\frac{n}{p}+\tau }\geq
\varepsilon ^{-\frac{n}{p}+\tau }.  \label{eq4.19}
\end{equation}%
We define $u$ by%
\begin{equation}
u^{p^{\ast }}=v^{p^{\ast }}+\sum_{j=1}^{l}\beta _{j}\psi
_{j}+c_{1}\varepsilon ^{-\frac{n}{p}+\tau }.  \label{eq4.20}
\end{equation}%
This $u$ will be the test function we use to prove Lemma \ref{lem4.1}.

It follows from (\ref{eq4.17}) that $\int_{\mathbb{S}^{n}}u^{p^{\ast }}fd\mu
=0$ for all $f\in \overset{\circ }{\mathcal{P}}_{2}$. Moreover using (\ref%
{eq4.11}) and (\ref{eq4.18}) we see%
\begin{eqnarray}
\int_{\mathbb{S}^{n}}u^{p^{\ast }}d\mu &=&\varepsilon ^{-\frac{n}{p}%
}\left\vert \mathbb{S}^{n-1}\right\vert \frac{n+2}{p^{\prime }}B\left( \frac{%
n}{p},\frac{n}{p^{\prime }}\right) +O\left( \varepsilon ^{-\frac{n}{p}+\tau
}\right)  \label{eq4.21} \\
&=&\varepsilon ^{-\frac{n}{p}}\left\vert \mathbb{S}^{n-1}\right\vert \frac{%
n+2}{p^{\prime }}B\left( \frac{n}{p},\frac{n}{p^{\prime }}\right) \left(
1+o\left( 1\right) \right)  \notag
\end{eqnarray}%
hence%
\begin{equation}
\left\Vert u\right\Vert _{L^{p^{\ast }}}^{p}=\varepsilon ^{-\frac{n}{p^{\ast
}}}\left\vert \mathbb{S}^{n-1}\right\vert ^{\frac{p}{p^{\ast }}}\left( \frac{%
n+2}{p^{\prime }}\right) ^{\frac{p}{p\ast }}B\left( \frac{n}{p},\frac{n}{%
p^{\prime }}\right) ^{\frac{p}{p^{\ast }}}\left( 1+o\left( 1\right) \right)
\label{eq4.22}
\end{equation}%
as $\varepsilon \rightarrow 0^{+}$. On the other hand, by (\ref{eq4.18}) and
(\ref{eq4.20}) we get%
\begin{equation}
u^{p^{\ast }}\leq v^{p^{\ast }}+c\varepsilon ^{-\frac{n}{p}+\tau },
\label{eq4.23}
\end{equation}%
hence%
\begin{equation}
u^{p}\leq \left( v^{p^{\ast }}+c\varepsilon ^{-\frac{n}{p}+\tau }\right) ^{%
\frac{p}{p^{\ast }}}\leq v^{p}+c\varepsilon ^{-\frac{n}{p^{\ast }}+\frac{p}{%
p^{\ast }}\tau }.  \label{eq4.24}
\end{equation}%
It follows that%
\begin{eqnarray}
&&\int_{\mathbb{S}^{n}}u^{p}d\mu  \label{eq4.25} \\
&\leq &\int_{\mathbb{S}^{n}}v^{p}d\mu +c\varepsilon ^{-\frac{n}{p^{\ast }}+%
\frac{p}{p^{\ast }}\tau }  \notag \\
&=&\left( n+2\right) \left\vert \mathbb{S}^{n-1}\right\vert
\int_{0}^{2\delta }\phi _{\varepsilon }\left( r\right) ^{p}\sin
^{n-1}rdr+c\varepsilon ^{-\frac{n}{p^{\ast }}+\frac{p}{p^{\ast }}\tau } 
\notag \\
&\leq &c\varepsilon ^{-\frac{n}{p^{\ast }}+\frac{p}{p^{\ast }}\tau }.  \notag
\end{eqnarray}%
Here we have used the choice of $\tau $ in (\ref{eq4.12}). Hence%
\begin{equation}
\left\Vert u\right\Vert _{L^{p}}^{p}=O\left( \varepsilon ^{-\frac{n}{p^{\ast
}}+\frac{p}{p^{\ast }}\tau }\right) =o\left( \varepsilon ^{-\frac{n}{p^{\ast
}}}\right)  \label{eq4.26}
\end{equation}%
as $\varepsilon \rightarrow 0^{+}$.

Next we note that%
\begin{eqnarray}
&&\left\Vert \nabla u\right\Vert _{L^{p}}^{p}  \label{eq4.27} \\
&=&\sum_{i=1}^{n+2}\int_{B_{\delta }\left( x_{i}\right) }\left\vert \nabla
u\right\vert ^{p}d\mu +O\left( \varepsilon ^{-\frac{n}{p^{\ast }}+\frac{p}{%
p^{\ast }}\tau }\right)  \notag \\
&=&\left( n+2\right) \left\vert \mathbb{S}^{n-1}\right\vert \left( \frac{%
np^{\prime }}{p^{\ast }}\right) ^{p}\int_{0}^{\delta }\frac{\left(
\varepsilon +r^{p^{\prime }}\right) ^{-n}r^{p^{\prime }}}{\left[
1+c_{1}\varepsilon ^{-\frac{n}{p}+\tau }\left( \varepsilon +r^{p^{\prime
}}\right) ^{n}\right] ^{p-\frac{p}{p^{\ast }}}}\sin ^{n-1}rdr+o\left(
\varepsilon ^{-\frac{n}{p^{\ast }}}\right)  \notag \\
&=&\varepsilon ^{-\frac{n}{p\ast }}\left\vert \mathbb{S}^{n-1}\right\vert
\left( \frac{n-p}{p-1}\right) ^{p}\cdot \frac{n+2}{p^{\prime }}B\left( \frac{%
n}{p}-1,\frac{n}{p^{\prime }}+1\right) +o\left( \varepsilon ^{-\frac{n}{%
p^{\ast }}}\right) .  \notag
\end{eqnarray}

It follows from (\ref{eq4.2}), (\ref{eq4.22}), (\ref{eq4.26}) and (\ref%
{eq4.27}) that%
\begin{eqnarray}
&&\varepsilon ^{-\frac{n}{p^{\ast }}}\left\vert \mathbb{S}^{n-1}\right\vert
^{\frac{p}{p^{\ast }}}\left( \frac{n+2}{p^{\prime }}\right) ^{\frac{p}{p\ast 
}}B\left( \frac{n}{p},\frac{n}{p^{\prime }}\right) ^{\frac{p}{p^{\ast }}%
}\left( 1+o\left( 1\right) \right)  \label{eq4.28} \\
&\leq &a\varepsilon ^{-\frac{n}{p\ast }}\left\vert \mathbb{S}%
^{n-1}\right\vert \left( \frac{n-p}{p-1}\right) ^{p}\cdot \frac{n+2}{%
p^{\prime }}B\left( \frac{n}{p}-1,\frac{n}{p^{\prime }}+1\right) +o\left(
\varepsilon ^{-\frac{n}{p^{\ast }}}\right) .  \notag
\end{eqnarray}%
Dividing both sides by $\varepsilon ^{-\frac{n}{p^{\ast }}}$ and letting $%
\varepsilon \rightarrow 0^{+}$, we see%
\begin{eqnarray*}
a &\geq &\left\vert \mathbb{S}^{n-1}\right\vert ^{-\frac{p}{n}}\left( \frac{%
p-1}{n-p}\right) ^{p}\left( \frac{n+2}{p^{\prime }}\right) ^{-\frac{p}{n}}%
\frac{B\left( \frac{n}{p},\frac{n}{p^{\prime }}\right) ^{\frac{p}{p^{\ast }}}%
}{B\left( \frac{n}{p}-1,\frac{n}{p^{\prime }}+1\right) } \\
&=&\left( n+2\right) ^{-\frac{p}{n}}S_{n,p}^{p}.
\end{eqnarray*}%
Here we have used the identity (\ref{eq4.6}). Lemma \ref{lem4.1} follows.

\subsection{$W^{s,\frac{n}{s}}\left( \mathbb{S}^{n}\right) $ for even $s$%
\label{sec4.1}}

Let $s\in \mathbb{N}$ be even, $1<p<\frac{n}{s}$, then for any $\varphi \in
C_{c}^{\infty }\left( \mathbb{R}^{n}\right) $,%
\begin{equation}
\left\Vert \varphi \right\Vert _{L^{\frac{np}{n-sp}}}\leq c\left(
n,s,p\right) \left\Vert \Delta ^{\frac{s}{2}}\varphi \right\Vert _{L^{p}}.
\label{eq4.29}
\end{equation}%
We denote%
\begin{eqnarray}
S_{n,s,p} &=&\sup_{\varphi \in C_{c}^{\infty }\left( \mathbb{R}^{n}\right)
\backslash \left\{ 0\right\} }\frac{\left\Vert \varphi \right\Vert _{L^{%
\frac{np}{n-sp}}}}{\left\Vert \Delta ^{\frac{s}{2}}\varphi \right\Vert
_{L^{p}}}  \label{eq4.30} \\
&=&\sup \left\{ \frac{\left\Vert u\right\Vert _{L^{\frac{np}{n-sp}}}}{%
\left\Vert \Delta ^{\frac{s}{2}}u\right\Vert _{L^{p}}}:u\in L^{\frac{np}{n-sp%
}}\left( \mathbb{R}^{n}\right) \backslash \left\{ 0\right\} \text{ s.t. }%
\Delta ^{\frac{s}{2}}u\in L^{p}\left( \mathbb{R}^{n}\right) \right\} . 
\notag
\end{eqnarray}%
Let $\left( M^{n},g\right) $ be a smooth compact Riemannian manifold of
dimension $n$ and $1<p<\frac{n}{s}$. Then we have the Sobolev space $%
W^{s,p}\left( M\right) $ with norm%
\begin{equation}
\left\Vert u\right\Vert _{W^{s,p}\left( M\right) }=\left(
\sum_{k=0}^{s}\left\Vert D^{k}u\right\Vert _{L^{p}\left( M\right)
}^{p}\right) ^{\frac{1}{p}}.  \label{eq4.31}
\end{equation}%
Here $D^{k}u$ is the covariant derivative of $u$ associated with the metric $%
g$. We have Aubin's almost sharp inequality: for any $\varepsilon >0$,%
\begin{equation}
\left\Vert u\right\Vert _{L^{\frac{np}{n-sp}}\left( M\right) }^{p}\leq
\left( S_{n,s,p}^{p}+\varepsilon \right) \left\Vert \Delta ^{\frac{s}{2}%
}u\right\Vert _{L^{p}\left( M\right) }^{p}+c\left( \varepsilon \right)
\left\Vert u\right\Vert _{L^{p}\left( M\right) }^{p}  \label{eq4.32}
\end{equation}%
for $u\in W^{s,p}\left( M\right) $ and the associated concentration
compactness principle (see \cite{DHL, He2, Ln1, Ln2}).

\begin{proposition}
\label{prop4.1}Let $\left( M^{n},g\right) $ be a smooth compact Riemannian
manifold of dimension $n$, $s\in \mathbb{N}$ be even and $1<p<\frac{n}{s}$.
If $u_{i}\in W^{s,p}\left( M\right) $ s.t. $u_{i}\rightharpoonup u$ weakly
in $W^{s,p}\left( M\right) $,%
\begin{eqnarray*}
\left\vert \Delta ^{\frac{s}{2}}u_{i}\right\vert ^{p}d\mu &\rightarrow
&\left\vert \Delta ^{\frac{s}{2}}u\right\vert ^{p}d\mu +\sigma \text{ as
measure,} \\
\left\vert u_{i}\right\vert ^{\frac{np}{n-sp}}d\mu &\rightarrow &\left\vert
u\right\vert ^{\frac{np}{n-sp}}d\mu +\nu \text{ as measure,}
\end{eqnarray*}%
here $\mu $ is the measure associated with $g$, then we can find countably
many points $x_{i}\in M$ s.t.%
\begin{eqnarray}
\nu &=&\sum_{i}\nu _{i}\delta _{x_{i}},  \label{eq4.33} \\
\nu _{i}^{\frac{1}{p}-\frac{s}{n}} &\leq &S_{n,s,p}\sigma _{i}^{\frac{1}{p}},
\label{eq4.34}
\end{eqnarray}%
here $\nu _{i}=\nu \left( \left\{ x_{i}\right\} \right) $, $\sigma
_{i}=\sigma \left( \left\{ x_{i}\right\} \right) $.
\end{proposition}

With Proposition \ref{prop4.1} in place of Theorem \ref{thm2.1}, we can
argue the same way as in the proof of Theorem \ref{thm2.2} to get

\begin{theorem}
\label{thm4.1}Assume $n\in \mathbb{N}$, $s\in \mathbb{N}$ is even, $1<p<%
\frac{n}{s}$ and $m\in \mathbb{N}$. Then for any $\varepsilon >0$, and $u\in
W^{s,p}\left( \mathbb{S}^{n}\right) $ with%
\begin{equation}
\left\vert \int_{\mathbb{S}^{n}}f\left\vert u\right\vert ^{\frac{np}{n-sp}%
}d\mu \right\vert \leq b\left( f\right) \left\Vert u\right\Vert _{L^{p}}^{%
\frac{np}{n-sp}}  \label{eq4.35}
\end{equation}%
for all $f\in \overset{\circ }{\mathcal{P}}_{m}$ (defined in (\ref{eq1.11}%
)), here $b\left( f\right) $ is a positive number depending on $f$, we have%
\begin{equation}
\left\Vert u\right\Vert _{L^{\frac{np}{n-sp}}\left( \mathbb{S}^{n}\right)
}^{p}\leq \left( \frac{S_{n,s,p}^{p}}{\Theta \left( m,\frac{n-sp}{n}%
,n\right) }+\varepsilon \right) \left\Vert \Delta ^{\frac{s}{2}}u\right\Vert
_{L^{p}\left( \mathbb{S}^{n}\right) }^{p}+c\left( \varepsilon ,b,m\right)
\left\Vert u\right\Vert _{L^{p}\left( \mathbb{S}^{n}\right) }^{p}.
\label{eq4.36}
\end{equation}%
Here $S_{n,s,p}$ is given by (\ref{eq4.30}), $\Theta \left( m,\frac{n-sp}{n}%
,n\right) $ is defined in (\ref{eq1.13}).
\end{theorem}

Combining Theorem \ref{thm4.1} with Proposition \ref{prop3.1} we get

\begin{corollary}
\label{cor4.1}Assume $n\in \mathbb{N}$, $s\in \mathbb{N}$ is even and $1<p<%
\frac{n}{s}$. If $u\in W^{s,p}\left( \mathbb{S}^{n}\right) $ satisfies%
\begin{equation}
\int_{\mathbb{S}^{n}}x_{i}\left\vert u\right\vert ^{\frac{np}{n-sp}}d\mu
\left( x\right) =0\text{ for }i=1,2,\cdots ,n+1,  \label{eq4.37}
\end{equation}%
here $\mu $ is the standard measure on $\mathbb{S}^{n}$, then for any $%
\varepsilon >0$, we have%
\begin{equation}
\left\Vert u\right\Vert _{L^{\frac{np}{n-sp}}\left( \mathbb{S}^{n}\right)
}^{p}\leq \left( 2^{-\frac{sp}{n}}S_{n,s,p}^{p}+\varepsilon \right)
\left\Vert \Delta ^{\frac{s}{2}}u\right\Vert _{L^{p}\left( \mathbb{S}%
^{n}\right) }^{p}+c_{\varepsilon }\left\Vert u\right\Vert _{L^{p}\left( 
\mathbb{S}^{n}\right) }^{p}.  \label{eq4.38}
\end{equation}
\end{corollary}

Combining Theorem \ref{thm4.1} with Proposition \ref{prop3.2} we get

\begin{corollary}
\label{cor4.2}Assume $n\in \mathbb{N}$, $s\in \mathbb{N}$ is even and $1<p<%
\frac{n}{s}$. If $u\in W^{s,p}\left( \mathbb{S}^{n}\right) $ satisfies%
\begin{equation}
\int_{\mathbb{S}^{n}}f\left\vert u\right\vert ^{\frac{np}{n-sp}}d\mu =0
\label{eq4.39}
\end{equation}%
for all $f\in \overset{\circ }{\mathcal{P}}_{2}$ (defined in (\ref{eq1.11}%
)), here $\mu $ is the standard measure on $\mathbb{S}^{n}$, then for any $%
\varepsilon >0$, we have%
\begin{equation}
\left\Vert u\right\Vert _{L^{\frac{np}{n-sp}}\left( \mathbb{S}^{n}\right)
}^{p}\leq \left( \left( n+2\right) ^{-\frac{sp}{n}}S_{n,s,p}^{p}+\varepsilon
\right) \left\Vert \Delta ^{\frac{s}{2}}u\right\Vert _{L^{p}\left( \mathbb{S}%
^{n}\right) }^{p}+c_{\varepsilon }\left\Vert u\right\Vert _{L^{p}\left( 
\mathbb{S}^{n}\right) }^{p}.  \label{eq4.40}
\end{equation}
\end{corollary}

\subsection{Fourth order Paneitz operator and the $Q$ curvature\label{sec4.2}%
}

Let $\left( M,g\right) $ be a smooth Riemannian manifold with dimension $%
n\geq 3$, the fourth order $Q$ curvature is given by (see \cite{B})%
\begin{equation}
Q=-\frac{1}{2\left( n-1\right) }\Delta R-\frac{2}{\left( n-2\right) ^{2}}%
\left\vert Rc\right\vert ^{2}+\frac{n^{3}-4n^{2}+16n-16}{8\left( n-1\right)
^{2}\left( n-2\right) ^{2}}R^{2}.  \label{eq4.41}
\end{equation}%
The fourth order Paneitz operator is given by%
\begin{eqnarray}
&&Pu  \label{eq4.42} \\
&=&\Delta ^{2}u+\frac{4}{n-2}\func{div}\left( Rc\left( \nabla u,e_{i}\right)
e_{i}\right) -\frac{n^{2}-4n+8}{2\left( n-1\right) \left( n-2\right) }\func{%
div}\left( R\nabla u\right) +\frac{n-4}{2}Qu.  \notag
\end{eqnarray}%
Here $e_{1},\cdots ,e_{n}$ is a local orthonormal frame with respect to $g$.
The Paneitz operator and $Q$ curvature behave the same way as the conformal
Laplacian operator and scalar curvature. Recent work in \cite{GuHL, GuM,
HaY1, HaY2, HaY3, HaY4} helps us a lot in the understanding of this special
fourth order operator in dimensions other than $4$. We refer the reader to 
\cite{HaY4} for a detailed discussion. For the standard $\mathbb{S}^{n}$,%
\begin{equation}
Pu=\Delta ^{2}u-\frac{n^{2}-2n-4}{2}\Delta u+\frac{n\left( n+2\right) \left(
n-2\right) \left( n-4\right) }{16}u.  \label{eq4.43}
\end{equation}

On the other hand, it follows from \cite{CnLO, Li, Lin} that for $n\geq 5$,%
\begin{eqnarray}
S_{n,2,2} &=&\sup_{\varphi \in C_{c}^{\infty }\left( \mathbb{R}^{n}\right)
\backslash \left\{ 0\right\} }\frac{\left\Vert \varphi \right\Vert _{L^{%
\frac{2n}{n-4}}}}{\left\Vert \Delta ^{2}\varphi \right\Vert _{L^{2}}}
\label{eq4.44} \\
&=&\sup \left\{ \frac{\left\Vert u\right\Vert _{L^{\frac{2n}{n-4}}}}{%
\left\Vert \Delta ^{2}u\right\Vert _{L^{2}}}:u\in L^{\frac{2n}{n-4}}\left( 
\mathbb{R}^{n}\right) \backslash \left\{ 0\right\} \text{ s.t. }\Delta
^{2}u\in L^{2}\left( \mathbb{R}^{n}\right) \right\}  \notag \\
&=&\frac{4}{\sqrt{n\left( n+2\right) \left( n-2\right) \left( n-4\right) }}%
\left\vert S^{n}\right\vert ^{-\frac{2}{n}},  \notag
\end{eqnarray}%
and $S_{n,2,2}$ is achieved if and only if $u\left( x\right) =\pm \left(
a+b\left\vert x-x_{0}\right\vert ^{2}\right) ^{-\frac{n-4}{2}}$ for some $%
x_{0}\in \mathbb{R}^{n}$ and positive constants $a$ and $b$. It follows from
Theorem \ref{thm4.1} that

\begin{corollary}
\label{cor4.3}Assume $n\geq 5$. If $u\in H^{2}\left( \mathbb{S}^{n}\right)
=W^{2,2}\left( \mathbb{S}^{n}\right) $ satisfies%
\begin{equation}
\int_{\mathbb{S}^{n}}x_{i}\left\vert u\right\vert ^{\frac{2n}{n-4}}d\mu
\left( x\right) =0\text{ for }i=1,2,\cdots ,n+1,  \label{eq4.45}
\end{equation}%
here $\mu $ is the standard measure on $\mathbb{S}^{n}$, then for any $%
\varepsilon >0$, we have%
\begin{equation}
\left\Vert u\right\Vert _{L^{\frac{2n}{n-4}}\left( \mathbb{S}^{n}\right)
}^{2}\leq \left( 2^{-\frac{4}{n}}S_{n,2,2}^{2}+\varepsilon \right)
\left\Vert \Delta u\right\Vert _{L^{2}\left( \mathbb{S}^{n}\right)
}^{2}+c_{\varepsilon }\left\Vert u\right\Vert _{L^{2}\left( \mathbb{S}%
^{n}\right) }^{2}.  \label{eq4.46}
\end{equation}%
Here $S_{n,2,2}$ is given in (\ref{eq4.44}).
\end{corollary}

Corollary \ref{cor4.3} can be used in the prescribing $Q$ curvature problem
on $\mathbb{S}^{n}$.

\subsection{$W^{s,\frac{n}{s}}\left( \mathbb{S}^{n}\right) $ for odd $s$%
\label{sec4.3}}

Let $s\in \mathbb{N}$ be odd, $1<p<\frac{n}{s}$, then for any $\varphi \in
C_{c}^{\infty }\left( \mathbb{R}^{n}\right) $,%
\begin{equation}
\left\Vert \varphi \right\Vert _{L^{\frac{np}{n-sp}}}\leq c\left(
n,s,p\right) \left\Vert \nabla \Delta ^{\frac{s-1}{2}}\varphi \right\Vert
_{L^{p}}.  \label{eq4.47}
\end{equation}%
We write%
\begin{eqnarray}
S_{n,s,p} &=&\sup_{\varphi \in C_{c}^{\infty }\left( \mathbb{R}^{n}\right)
\backslash \left\{ 0\right\} }\frac{\left\Vert \varphi \right\Vert _{L^{%
\frac{np}{n-sp}}}}{\left\Vert \nabla \Delta ^{\frac{s-1}{2}}\varphi
\right\Vert _{L^{p}}}  \label{eq4.48} \\
&=&\sup \left\{ \frac{\left\Vert u\right\Vert _{L^{\frac{np}{n-sp}}}}{%
\left\Vert \nabla \Delta ^{\frac{s-1}{2}}u\right\Vert _{L^{p}}}:u\in L^{%
\frac{np}{n-sp}}\left( \mathbb{R}^{n}\right) \backslash \left\{ 0\right\} 
\text{ s.t. }\nabla \Delta ^{\frac{s-1}{2}}u\in L^{p}\left( \mathbb{R}%
^{n}\right) \right\} .  \notag
\end{eqnarray}%
We can proceed in the same way as in Section \ref{sec4.2} to get the
following

\begin{theorem}
\label{thm4.2}Assume $n\in \mathbb{N}$, $s\in \mathbb{N}$ is odd, $1<p<\frac{%
n}{s}$ and $m\in \mathbb{N}$. Then for any $\varepsilon >0$, and $u\in
W^{s,p}\left( \mathbb{S}^{n}\right) $ with%
\begin{equation}
\left\vert \int_{\mathbb{S}^{n}}f\left\vert u\right\vert ^{\frac{np}{n-sp}%
}d\mu \right\vert \leq b\left( f\right) \left\Vert u\right\Vert _{L^{p}}^{%
\frac{np}{n-sp}}  \label{eq4.49}
\end{equation}%
for all $f\in \overset{\circ }{\mathcal{P}}_{m}$ (defined in (\ref{eq1.11}%
)), here $b\left( f\right) $ is a positive number depending on $f$, we have%
\begin{equation}
\left\Vert u\right\Vert _{L^{\frac{np}{n-sp}}\left( \mathbb{S}^{n}\right)
}^{p}\leq \left( \frac{S_{n,s,p}^{p}}{\Theta \left( m,\frac{n-sp}{n}%
,n\right) }+\varepsilon \right) \left\Vert \nabla \Delta ^{\frac{s-1}{2}%
}u\right\Vert _{L^{p}\left( \mathbb{S}^{n}\right) }^{p}+c\left( \varepsilon
,b,m\right) \left\Vert u\right\Vert _{L^{p}\left( \mathbb{S}^{n}\right)
}^{p}.  \label{eq4.50}
\end{equation}%
Here $S_{n,s,p}$ is given by (\ref{eq4.48}), $\Theta \left( m,\frac{n-sp}{n}%
,n\right) $ is defined in (\ref{eq1.13}).
\end{theorem}

\begin{corollary}
\label{cor4.4}Assume $n\in \mathbb{N}$, $s\in \mathbb{N}$ is odd and $1<p<%
\frac{n}{s}$. If $u\in W^{s,p}\left( \mathbb{S}^{n}\right) $ satisfies%
\begin{equation}
\int_{\mathbb{S}^{n}}x_{i}\left\vert u\right\vert ^{\frac{np}{n-sp}}d\mu
\left( x\right) =0\text{ for }i=1,2,\cdots ,n+1,  \label{eq4.51}
\end{equation}%
here $\mu $ is the standard measure on $\mathbb{S}^{n}$, then for any $%
\varepsilon >0$, we have%
\begin{equation}
\left\Vert u\right\Vert _{L^{\frac{np}{n-sp}}\left( \mathbb{S}^{n}\right)
}^{p}\leq \left( 2^{-\frac{sp}{n}}S_{n,s,p}^{p}+\varepsilon \right)
\left\Vert \nabla \Delta ^{\frac{s-1}{2}}u\right\Vert _{L^{p}\left( \mathbb{S%
}^{n}\right) }^{p}+c_{\varepsilon }\left\Vert u\right\Vert _{L^{p}\left( 
\mathbb{S}^{n}\right) }^{p}.  \label{eq4.52}
\end{equation}
\end{corollary}

\begin{corollary}
\label{cor4.5}Assume $n\in \mathbb{N}$, $s\in \mathbb{N}$ is odd and $1<p<%
\frac{n}{s}$. If $u\in W^{s,p}\left( \mathbb{S}^{n}\right) $ satisfies%
\begin{equation}
\int_{\mathbb{S}^{n}}f\left\vert u\right\vert ^{\frac{np}{n-sp}}d\mu =0
\label{eq4.53}
\end{equation}%
for all $f\in \overset{\circ }{\mathcal{P}}_{2}$ (defined in (\ref{eq1.11}%
)), then for any $\varepsilon >0$, we have%
\begin{equation}
\left\Vert u\right\Vert _{L^{\frac{np}{n-sp}}\left( \mathbb{S}^{n}\right)
}^{p}\leq \left( \left( n+2\right) ^{-\frac{sp}{n}}S_{n,s,p}^{p}+\varepsilon
\right) \left\Vert \nabla \Delta ^{\frac{s-1}{2}}u\right\Vert _{L^{p}\left( 
\mathbb{S}^{n}\right) }^{p}+c_{\varepsilon }\left\Vert u\right\Vert
_{L^{p}\left( \mathbb{S}^{n}\right) }^{p}.  \label{eq4.54}
\end{equation}
\end{corollary}

We leave the detail to interested readers.

\textit{Note added in proof. }Recently it is shown in \cite{P} that%
\begin{equation}
\Theta \left( 3,\theta ,n\right) =\left( 2n+2\right) ^{1-\theta }
\label{eq4.55}
\end{equation}%
and $\Theta \left( 3,\theta ,n\right) $ is achieved at $\nu \in \mathcal{M}%
_{3}^{c}\left( \mathbb{S}^{n}\right) $ if and only if%
\begin{equation}
\nu =\frac{1}{2n+2}\sum_{i=1}^{n+1}\left( \delta _{\xi _{i}}+\delta _{-\xi
_{i}}\right)  \label{eq4.56}
\end{equation}%
for $\xi _{1},\cdots ,\xi _{n+1}$ being an orthonormal basis of $\mathbb{R}%
^{n+1}$.

\end{document}